\documentclass[journal,twoside]{IEEEtran} 

\usepackage[english]{babel}
\usepackage[noend]{algorithmic}
\usepackage{algorithm}
\usepackage{amsmath,amssymb,amsthm,amsbsy,latexsym}
\usepackage{mathrsfs}
\usepackage{multirow}
\usepackage{enumerate}
\usepackage{color}
\usepackage{cite}

\newtheorem{theorem}{Theorem}
\newtheorem{remark}{Remark}

\newtheorem{definition}{Definition}
\newtheorem{lemma}{Lemma}
\newtheorem{example}{Example}

\newtheorem{notation}{Notation}
\newtheorem{corollary}{Corollary}
\newtheorem{rem}{Remark}
\newtheorem{Procedure}{Procedure}

\newcommand{\QNUM}{n_\mathrm{p}}
\newcommand{\SWS}{\mathfrak{S}}
\newcommand{\LPV}{\mathfrak{LPV}}

\newcommand{\Rank}{\mathrm{rank}}
\newcommand{\rank}{\mathrm{rank}}

\newcommand{\SPAN}{\mathrm{span}}
\newcommand{\NX}{n_\mathrm{x}}
\newcommand{\NY}{n_\mathrm{y}}
\newcommand{\NU}{n_\mathrm{u}}

\newcommand{\X}{\mathbb{R}^{n_\mathrm{x}}}

\newcommand{\Words}{\mathcal{S}(\mathbb{I}_0^{n_\mathrm{p}})}
\newcommand{\AQ}{\mathbb{I}_0^{n_\mathrm{p}}}

\begin{document}
\title{Realization Theory for LPV State-Space Representations with Affine Dependence}
\author{Mih\'aly Petreczky ~\IEEEmembership{Member,~IEEE}, Roland T\'oth,~\IEEEmembership{Member,~IEEE} and Guillaume Merc\`{e}re ~\IEEEmembership{Member,~IEEE} 
\thanks{Mih\'aly Petreczky (Corresponding author) s with Centre de Recherche en Informatique, Signal et Automatique de Lille (CRIStAL) {\tt\small mihaly.petreczky@ec-lille.fr}}
\thanks{Guillaume Merc\`{e}re is with the University of Poitiers, Laboratoire d'Informatique et d'Automatique pour les Syst\`emes, 2 rue P. Brousse, batiment B25, B.P. 633, 86022 Poitiers Cedex, France. Email: {\tt guillaume.mercere@univ-poitiers.fr}}
\thanks{R. T\'oth is with the Control Systems Group, Department of Electrical Engineering, Eindhoven University of Technology, P.O. Box 513, 5600 MB Eindhoven, The Netherlands. Email: {\tt r.toth@tue.nl}.}
\thanks{This work was partially supported by ESTIREZ project of Region Nord-Pas de Calais, France}}

\maketitle
\begin{abstract}
  In this paper we present a Kalman-style realization theory for linear parameter-varying state-space representations whose matrices depend on the scheduling variables in an affine way (abbreviated as LPV-SSA representations). 
  We show that such a LPV-SSA representation 
  is a minimal (in the sense of having the least number of state-variables) representation of its input-output function, if and only if it is observable and span-reachable. We show that any two minimal 
  LPV-SSA representations of the same input-output function are related by a linear isomorphism, and the isomorphism does not depend on the scheduling variable. 
  We show that an input-output function can be represented by a LPV-SSA representation if and only if the Hankel-matrix of the input-output function has a finite rank. In fact, the rank of the Hankel-matrix gives the dimension of a minimal
  LPV-SSA representation. Moreover, we can formulate a counterpart of partial realization theory for LPV-SSA representation and prove correctness of the Kalman-Ho algorithm
  formulated in \cite{TAW12}. These results thus represent the basis of systems theory for LPV-SSA representation. 
\end{abstract}

\section{Introduction}

\emph{Linear parameter-varying} (LPV) systems represent an intermediate system class between the class of \emph{linear time-invariant} (LTI) systems and systems with \emph{nonlinear} and \emph{time-varying} behavior.  
The underlying idea behind the use of LPV systems is to approximately model nonlinear and time-varying systems by 
linear time-varying difference or differential equations, where the time varying coefficients are 
functions of a time-varying signal, the
so-called \emph{scheduling variable}.  Such equations are called LPV systems \cite{Toth2010SpringerBook,Toth11_LPVBehav}. 
That is, LPV systems are a class of mathematical models having a certain structure (linear and time-varying).
The use of LPV systems is motivated by the fact that control
design for these systems is well developed \cite{Rugh00,Packard94CL,Apkarian95TACT,Carsten96,Lu2004,Wu2006,Scherer2009}. 
More recently, system identification of LPV models has gained attention
\cite{CHToth2012,Toth10_SRIV,Giarre2002,Butcher08,Poolla2008,Wingerden09,LopesDosSantos2008,Sznaier:01,Verdult02,Val13,Val14}.

Despite these advances and the popularity of LPV models, there are significant gaps in their systems theory, in particular, their realization theory.   By realization theory we mean a systematic characterization of the relationship between the input-output behavior of LPV systems and their state-space representations.
More precisely, 
we will make a distinction between LPV state-space representations, which are mathematical models in terms of difference/differential equations, and their input-output behavior (i.e. the set of input-output trajectories which they generate). 
The question realization
theory tries to answer is how to characterize those LPV state-space representations which describe the same set of
input-output trajectories, and how to construct such an LPV state-space representation from the set of input-output trajectories.  The reason that this problem is an important one is as follows.
Notice that it in general, there is no justification to claim that a designated LPV state-space 
representation is the `true" model of the  physical phenomenon of interest. Any other LPV state-space representation which generates the same input-output trajectories as this designated state-space representation can also be viewed as a model of the
physical phenomenon. For example, two different system identification techniques or modelling approaches could yield two
different state-space representations which are equivalent, in the sense that they describe the same
set of input-output trajectories.   In this case, there is no reason to prefer one model over the other one. 
Hence, any controller developed using one such LPV state-space 
representation should be shown to work for all the other LPV state-space representation generating the same input-output behavior.
In order to address this issue, we need realization theory.  
As for system identification, the best we can hope for a system identification algorithm is that it will find one LPV state-space representation which generates (at least approximately, with some error) the observed input-output trajectories. 
Hence, we need realization theory for analyzing system identification algorithms, as it tells us the set of possible correct
outcomes of any system identification algorithm under ideal circumstances (no noise, etc.). 
The same goes for analyzing model reduction algorithms. 
Moreover, from a practical point of view,  what we are interested in is the interplay between system identification, model
reduction and control design. Roughly speaking, we would like to know when we can hope that a  controller which was calculated based on a plant model obtained from system identification and model reduction algorithms will work for the original system. 
In order to understand this interplay, we need to understand the relationship between various LPV state-space representations which are consistent with the same input-output trajectories, i.e. we need realization theory.

The only systematic effort to address this gap was made in 
\cite{Toth2010SpringerBook,Toth11_LPVBehav}, where  behavioral theory was used to clarify realization theory, concepts of minimality and equivalence classes of various LPV representation forms.
However, the LPV models considered in \cite{Toth2010SpringerBook,Toth11_LPVBehav} assumed non-linear (meromorphic) and \emph{dynamical dependence} of the model parameters on the scheduling variable.
More precisely, the LPV model parameters were assumed to be meromorphic functions of the scheduling variable and its derivatives (in continuous-time), or of the current and future values of the scheduling variable
(discrete-time). As a result, the system theoretic transformations (passing from input-output behavior to state-space representation, transforming a state-space representation to a minimal one, etc.) described in
\cite{Toth2010SpringerBook,Toth11_LPVBehav} introduce LPV models with a dynamic and nonlinear dependence on the parameters. However, for practical applications it is preferable to use LPV models with a static and
affine dependence on the scheduling variable, i.e., LPV models whose system parameters are affine functions of the instantenous value of the scheduling variable.
That is, from a practical point of view it make sense to concentrate on systems theory of LPV models with static and affine dependence. 
In particular, the following fundamental question which directly pops up in the engineering context has remained unanswered: when is it possible to give a simple state-space model with affine static dependence for an identified or modeled LPV system behavior and how to accomplish this realization step with ease. To find an answer to this question is the main motivation of this paper. %


 
 In this paper we present a Kalman-like realization theory for LPV state-space representations with affine static dependence of coefficients, abbreviated as \emph{LPV-SSA representations}.
We will consider both the \emph{discrete-time} (DT) and the \emph{continuous-time} (CT) cases. In particular, we show existence, uniqueness of a special form of equivalent \emph{infinite impulse response representations (IIR)} of systems with such representations both in DT and CT cases. We show that all input-output functions which can be described by an LPV-SSA representation admit an IIR, and conversely, if an input-output function admits an IIR and its 
Hankel matrix  has finite rank, then this input-output function can be represented by an LPV-SSA representation. In this case, the finite rank of the Hankel matrix equals the dimension of a minimal LPV-SSA realization of
the input-output function.
Furthermore, the concept of (state) minimality and Kalman decomposition in terms of observability and reachability is clarified. It is proven that for the LPV-SSA case, state-minimality is equivalent with joint observability and span-reachability. It is shown that the  construction of a minimal LPV-SSA form of an arbitrary LPV input-output function can be always (if it exists) carried out with the application of the Ho-Kalman realization algorithm. 
We also discuss partial LPV-SSA realization for input-output functions.
 Moreover, it is formally proven that all minimal LPV-SSA representations of the same input-output function are isomorphic, and this isomorphism is linear and it does not depend on the scheduling parameter. 
Finally, we show that under some very mild conditions,  minimal LPV-SSA representations are also minimal among the meromorphic LPV state-space representations from \cite{Tot10}.
 The results of this paper could be useful for model reduction and system identification of LPV-SSAs. For example, they could be useful for improved subspace identification algorithm (see \cite{CoxSubspace} for preliminary results), for identifiability analysis, characterization of identifiability, topology of minimal LPV-SSAs, identifiable canonical forms (following the idea of \cite{MP:HSCC2010,Han89,PP10}), for finding conditions for persistence of excitation of LPV-SSAs (following the ideas of \cite{MPLB:Pers})  or  for model reduction using moment matching (see \cite{MertLPVCDC2015} for preliminary results). 

Many of concepts related to those used in this paper have already been published in various works, but without the existence of a coherent connection and underlying formal proofs. 
In particular, the idea of Hankel-matrix has appeared in \cite{Wingerden09,Verdult02,Toth10_SSReal}.  
 The Markov-parameters and the realization algorithm 
were already described in \cite{Toth10_SSReal}. 
In contrast to \cite{Wingerden09,Verdult02,Toth10_SSReal}, in this paper, the Markov-parameters 
and the related Hankel-matrix are defined directly for input-output functions, 
without assuming the existence of a finite dimensional LPV-SSA realization. 
In fact, the finite rank of the Hankel-matrix represents the necessary 
and sufficient condition for the existence of an LPV-SSA realization. 
In addition, we discuss the conditions for the correctness of the realization algorithm
in more details. 
Extended observability and reachability matrices were also presented in \cite{Toth2010SpringerBook,Rugh96book}. 
However, their system-theoretic interpretation as well as the relationship with minimality were not explored. 
Realization theory of more general linear parameter-varying systems was already developed in \cite{Toth2010SpringerBook}, 
the system matrices were allowed to depend on the scheduling parameters in a non-linear way, however, the results 
published in \cite{Toth2010SpringerBook}  do not always imply the ones, for the restricted LPV-SSA case, presented in this paper. Furthermore, the results presented in this paper can be also seen as generalization of system theoretical results available for linear switched systems \cite{Pet07,Pet11,Pet12,PetCocv11}. 
The current paper is partially based on \cite{PM12}. With respect to \cite{PM12}, the main differences are as follows. First, \cite{PM12} presents the results without proofs. 
Second, \cite{PM12} deals only with the DT case, while the extension of the results to  CT case is challenging and technically more involved than the DT case.  
Finally, the exposition has been improved and simplified in comparison to \cite{PM12}.  The technical report
\cite{MPGM:LPVTech} differs from \cite{PM12} only in the presence of some sketches of the proofs for the results of \cite{PM12}. 

The paper is organized as follows: In Section \ref{para:formulation}, basic notions and concepts are introduced, which is followed, in Section \ref{sec:pre}, by the definition of SSA representations, input-output functions, equivalence and minimality in the considered LPV context.  In Section \ref{para:main}, the main results of the paper in terms of existence, uniqueness and convergence of SSA inducing impulse response representations and the corresponding concepts of Hankel matrix and SSA realization theory are explained. For the sake of readability, all proofs are collected in Appendix \ref{para:proof}. 

\section{Notation}\label{para:formulation}

The following notation is used: for a (possibly infinite) set $X$, denote by $\mathcal{S}(X)$ the set of finite sequences generated from $X$, \emph{i.e.}, each $s \in \mathcal{S}(X)$ is of the form $s=\zeta_{1}\zeta_{2} \cdots \zeta_{k}$ with $\zeta_1,\zeta_2,\ldots,\zeta_k \in X$, $k\in\mathbb{N}$. $|s|$ denotes the length of the sequence $s$, while for $s,r \in \mathcal{S}(X)$, $sr\in \mathcal{S}(X)$ corresponds the concatenation operation.  The symbol $\varepsilon$ is used for the empty sequence and $|\varepsilon|=0$ with $s\varepsilon=\varepsilon s=s$. Furthermore,  $X^{\mathbb{N}}$ denotes the set of all functions of the form $f:\mathbb{N} \rightarrow X$. For each $j=1,\ldots,m$, $e_j$ is the $j^\mathrm{th}$ \emph{standard basis} in $\mathbb{R}^{m}$.
Furthermore, let $\mathbb{I}_{s_1}^{s_2}=\{s\in\mathbb{Z}\mid s_1\leq s\leq s_2\}$ be an index set.



%

Let $\mathbb{T}=\mathbb{R}_{0}^{+}=[0,+\infty)$ be the time axis in the  \emph{continuous-time} (CT) case and $\mathbb{T}=\mathbb{N}$ in the  \emph{discrete-time} (DT) case. Denote by $\xi$ the differentiation operator $\frac{d}{dt}$ (in CT) and the forward time-shift
operator $q$ (in DT), \emph{i.e.},
if $z: \mathbb{T} \rightarrow \mathbb{R}^{n}$,  then
$(\xi z)(t)=\frac{d}{dt} z(t)$, if $\mathbb{T}=\mathbb{R}_{0}^{+}$, and
$(\xi z)(t)=z(t+1)$, if $\mathbb{T}=\mathbb{N}$.
As usual, denote by $\xi^k$ the $k$-fold application of $\xi$, i.e. for any $z:\mathbb{T} \rightarrow \mathbb{R}^n$,
$\xi^{0} z=z$, and $\xi^{k+1} z = \xi (\xi^k z)$ for all $k \in \mathbb{N}$.
Both for CT and DT, for any $\tau \in \mathbb{T}$, define the
time shift operator $q^{\tau}$ as follows: for any $f:\mathbb{T} \rightarrow \mathbb{R}^n$, $q^{\tau}f:\mathbb{T} \rightarrow \mathbb{R}^n$ is
defined by $(q^{\tau}f)(t)=f(t+\tau)$, $t \in \mathbb{T}$.

 A function $f:=\mathbb{R}_{0}^{+} \rightarrow \mathbb{R}^{n}$ is called
\emph{piecewise-continuous}, if $f$ has finitely many points of
discontinuity on any compact subinterval of $\mathbb{R}_{0}^{+}$ and, at any point of
discontinuity, the left-hand and right-hand side limits of $f$ exist and 
are finite.  We denote by $\mathcal{C}_\mathrm{p}(\mathbb{R}_{0}^{+},\mathbb{R}^n)$ the set of all $n$-dimensional
\emph{piecewise-continuous functions} of the above form. The notation
$\mathcal{C}_\mathrm{a}(\mathbb{R}_{0}^{+},\mathbb{R}^n)$ designates the set of all $n$-dimensional \emph{absolutely
  continuous functions}  \cite{LangRealAnalysisBook}.

Recall from \cite{WebsterBook} the following notions on affine hulls and affine bases. 
Recall that $b \in \mathbb{R}^n$ is an affine combination of $a_1,\ldots,a_N \in \mathbb{R}^n$, if $b=\sum_{i=1}^{N} \lambda_i a_i$ for some $\lambda_1,\ldots,\lambda_N \in \mathbb{R}$,$\sum_{i=1}^{N} \lambda_i=1$.
The \emph{affine hull} 
$\mathrm{Aff}~ A$ of a set $A$ is the set of affine combinations of elements of $A$. 
The vectors $b_1,\ldots,b_{m} \in \mathbb{R}^n$
are said to be affinely independent if for every $j=1,\ldots,m$, $b_j$ cannot be expressed as an affine combination of 
$\{b_i\}_{i=1, i \ne j}^{m}$. The vectors $b_1,\ldots,b_m$ are an affine basis of $\mathbb{R}^{n}$ if $m=n+1$,
$b_1,\ldots, b_{n+1}$ are affinely independent and $\mathrm{Aff}~ \{b_1,\ldots,b_{n+1}\}=\mathbb{R}^n$. 
\section{Preliminaries} \label{sec:pre}

In this paper, we consider the class of LPV systems that have LPV \emph{state-space} (SS) representations with \emph{affine} linear dependence on the \emph{scheduling variable}. We use the abbreviation
LPV-SSA to denote this subclass of state-space representations, defined as
\begin{equation}\label{equ:alpvss}
  \Sigma \ \left\{
  \begin{array}{lcl}
   \xi x (t) &=& A(p(t)) x(t) + B(p(t)) u(t), \\
   y(t)  &=& C(p(t)) x(t) + D(p(t))u(t),
  \end{array}\right.
\end{equation}
where $x(t) \in \mathbb{X}=\mathbb{R}^{n_\mathrm{x}}$ is the state variable, $y(t) \in \mathbb{Y}=\mathbb{R}^{n_\mathrm{y}}$ is the (measured) output, $u (t) \in \mathbb{U}=\mathbb{R}^{n_\mathrm{u}}$ represents the input signal and $p(t) \in \mathbb{P}\subseteq \mathbb{R}^{n_\mathrm{p}}$ is the so called \emph{scheduling variable} of the system represented by $\Sigma$, and 
 \begin{equation}\label{equ:affdep}
\begin{split} 
  A(\mathbf{p}) = A_0 + \sum_{i=1}^{\QNUM} A_i \mathbf{p}_i \mbox{, \ \  } 
  B(\mathbf{p}) = B_0 + \sum_{i=1}^{\QNUM} B_i \mathbf{p}_i,   \\
  C(\mathbf{p}) = C_0 + \sum_{i=1}^{\QNUM} C_i \mathbf{p}_i \mbox{, \ \ } 
  D(\mathbf{p}) = D_0 + \sum_{i=1}^{\QNUM} D_i \mathbf{p}_i,
\end{split}
\end{equation}
for every $\mathbf{p}=[\begin{array}{ccc} \mathbf{p}_1& \ldots & \mathbf{p}_{\QNUM}\end{array}]^\top \in \mathbb{P}$, with constant matrices $A_i \in \mathbb{R}^{\NX \times \NX}$,
$B_i \in \mathbb{R}^{\NX \times \NU}$, $C_i \in \mathbb{R}^{\NY \times \NX}$ and
$D_i \in \mathbb{R}^{\NY \times \NU}$ for all $i\in\mathbb{I}_0^{n_\mathrm{p}}$. 
\emph{It is assumed that $\mathrm{Aff}~ \mathbb{P}=\mathbb{R}^{n_\mathrm{p}}$, \emph{i.e.}, $\mathbb{P}$ contains an affine basis of $\mathbb{R}^{n_\mathrm{p}}$}, see Section \ref{para:formulation} or \cite{WebsterBook}
for the definition of affine span and affine basis. 
According to the LPV modeling concept, $p$ corresponds to varying-operating conditions, nonlinear/time-varying dynamical aspects and /or external effects influencing the plant behavior and it is allowed to vary in the  set $\mathbb{P}$, see \cite{Toth2010SpringerBook} for details. In the sequel, we will often use the shorthand notation \vspace{-1mm}
\begin{equation*}
   \Sigma=(\mathbb{P},\left\{ A_i, B_i, C_i, D_i \right\}_{i=0}^{\QNUM})
\end{equation*}
to denote an LPV-SSA representation of the form \eqref{equ:alpvss} and use  $\dim{( \Sigma )}=n_\mathrm{x}$ to denote its state dimension. 

 By a solution  of $\Sigma$ we mean a tuple of trajectories $(x,y,u,p)\in(\mathcal{X},\mathcal{Y},\mathcal{U},\mathcal{P})$ satisfying \eqref{equ:alpvss} for almost all $t \in \mathbb{T}$ in CT case, and for all $t \in \mathbb{T}$ in DT, where in CT, $\mathcal{X}=\mathcal{C}_\mathrm{a}(\mathbb{R}_{0}^{+},\mathbb{X}), \mathcal{Y}=\mathcal{C}_\mathrm{p}(\mathbb{R}_{0}^{+},\mathbb{Y}), \mathcal{U}=\mathcal{C}_\mathrm{p}(\mathbb{R}_{0}^{+},\mathbb{U}), \mathcal{P}=\mathcal{C}_\mathrm{p}(\mathbb{R}_{0}^{+},\mathbb{P})$, and in DT $\mathcal{X}=\mathbb{X}^\mathbb{N}, \mathcal{Y}=\mathbb{Y}^{\mathbb{N}}, \mathcal{U}=\mathbb{U}^{\mathbb{N}},\mathcal{P}=\mathbb{P}^{\mathbb{N}}$.

\begin{rem}[Zero initial time]
Notice that without loss of generality, the solution trajectories in CT can be considered on the half line $\mathbb{R}_0^+$ with
$t_\mathrm{o}=0$. Indeed, if $(x,y,u,p)$ satisfy 
\eqref{equ:alpvss}, then 
$(q^\tau x, q^\tau y, q^\tau u, q^\tau p)$
satisfies \eqref{equ:alpvss} for any $\tau\in\mathbb{R}$ (see \cite{Toth11_LPVBehav}).
Here $q^{\tau}$ is the shift operator defined in Section \ref{para:formulation}.
\end{rem}

Note that for any input and scheduling signal  $(u,p)\in\mathcal{U}\times \mathcal{P}$ and any initial state $x_{\mathrm o} \in \mathbb{X}$ , there exists a \emph{unique} pair $(y,x)\in\mathcal{Y}\times \mathcal{X}$ such that $(x,y,u,p)$ is a solution of  \eqref{equ:alpvss} and $x(0)=x_{\mathrm o}$,
see \cite{Toth2010SpringerBook}.
 That is, the dynamics of $\Sigma$ are thus driven by the inputs $u \in \mathcal{U}$ as well as the scheduling variables $p \in \mathcal{P}$.
This allows to define \emph{input-to-state}  and \emph{input-output}  functions as follows.
\begin{definition}[IS and IO functions]
  Let $x_\mathrm{o} \in \X$ be an initial state of $\Sigma$. Define the functions 
\begin{subequations}
\begin{align}
\mathfrak{X}_{\Sigma,x_\mathrm{o}} &:   \mathcal{U} \times \mathcal{P}  \rightarrow \mathcal{X}, \\
\mathfrak{Y}_{\Sigma,x_\mathrm{o}} &:   \mathcal{U} \times \mathcal{P}  \rightarrow \mathcal{Y}, \
\end{align}
\end{subequations}
such that for any $(x,y,u,p) \in \mathcal{X} \times \mathcal{Y} \times \mathcal{U} \times \mathcal{P}$,
$x=\mathfrak{X}_{\Sigma,x_\mathrm{o}}(u,p)$ and 
$y=\mathfrak{Y}_{\Sigma,x_\mathrm{o}}(u,p)$ holds if and only if
$(x,y,u,p)$ is a solution of \eqref{equ:alpvss} and $x(0)=x_{\mathrm o}$.
The function $\mathfrak{X}_{\Sigma,x_{\mathrm o}}$ is called the input-to-state function of $\Sigma$ induced by the initial state $x_{\mathrm o}$, and the function
$\mathfrak{Y}_{\Sigma,x_{\mathrm o}}$ is called the  input-to-output function $\Sigma$ induced by $x_{\mathrm o}$.
\end{definition}
Prompted by the definition above, we formalize potential input-output behavior of LPV-SSA representations 
as functions of the form
   \begin{equation}\label{equ:iofunction}
     \mathfrak{F} : \mathcal{U} \times \mathcal{P} \rightarrow \mathcal{Y}.
   \end{equation}
Note that an input-output map of any LPV-SSA representation is of the above form. However, not all maps of the form \eqref{equ:iofunction} arise as
input-output maps of some LPV-SSA representation.
%
\begin{definition}[Realization]
\label{def:real}
    The LPV-SSA representation $\Sigma$ is a realization of an input-output function $\mathfrak{F}$ of the form \eqref{equ:iofunction} from the initial state $x_{\mathrm o} \in \mathbb{X}$, 
  if  $\mathfrak{F}=\mathfrak{Y}_{\Sigma,x_\mathrm{o}}$; $\Sigma$ is said to be a realization of $\mathfrak{F}$, if there exist an initial
  state $x_{\mathrm o} \in \mathbb{X}$ of $\Sigma$, such that $\Sigma$ is a realization of $\mathfrak{F}$ from $x_{\mathrm o}$. 

\end{definition}
Similarly to \cite{Pet07,PetCocv11,Pet12}, the results of this paper could be extended to families of input-output functions with multiple initial states. However, in order to keep the notations simple, we only deal with systems having one initial state.
\begin{definition}[Input-output equivalence]
Two LPV-SSA representations  $\Sigma$ and $\Sigma^\prime$ are said to be weakly input-output equivalent w.r.t.\ the initial states $x\in\mathbb{R}^{n_\mathrm{x}}$ and $x^\prime \in\mathbb{R}^{n_{\mathrm{x}^\prime}}$, 
if $\mathfrak{Y}_{\Sigma,x}=\mathfrak{Y}_{\Sigma^\prime,x^\prime}$. They are called strongly input-output equivalent, if for all $x\in\mathbb{R}^{n_\mathrm{x}}$ there is a $x^\prime \in\mathbb{R}^{n_{\mathrm{x}^\prime}}$ such that $\mathfrak{Y}_{\Sigma,x}=\mathfrak{Y}_{\Sigma^\prime,x^\prime}$, and vice versa, for any  $x^\prime \in\mathbb{R}^{n_\mathrm{x}}$ there is a $x \in\mathbb{R}^{n_{\mathrm{x}^\prime}}$ such that  $\mathfrak{Y}_{\Sigma,x}=\mathfrak{Y}_{\Sigma^\prime,x^\prime}$.
\end{definition}
\begin{remark}[IO functions vs behaviors]
\label{rem:behavior}
So far, we formalized the input-output behavior of the system represented by $\Sigma$ as an input-output
 function induced by some initial state. 
 Another option is to use a behavioral approach, where the input-output (manifest) behavior of a given LPV-SSA $\Sigma$ is defined as
 \begin{multline} \mathfrak{B}(\Sigma)=\bigl\{(y,u,p) \in \mathcal{Y} \times \mathcal{U} \times \mathcal{P} \mid \exists x \in \mathcal{X} \\  \mbox{ s.t. } (x,y,u,p) \mbox{ satisfies \eqref{equ:alpvss}}
    \bigr\}.
 \end{multline}
Then, a $\Sigma$ realizes a $\mathfrak{B}\subseteq \mathcal{Y}\times \mathcal{U} \times \mathcal{P}$, if and only if $\mathfrak{B}=\mathfrak{B}(\Sigma)$.
Notice that $\mathfrak{B}(\Sigma)=\{ (y,u,p) \mid \exists x \in \X: \mathfrak{Y}_{\Sigma,x} (u,p)=y\}$, i.e., $\mathcal{B}(\Sigma)$ is just the union of graphs of the input-output functions $\mathfrak{Y}_{\Sigma,x}$.
 This prompts us to consider the following definition. Let $\Phi$ be a 
 set of input-output functions of the form 
 $\mathfrak{F}:\mathcal{U} \times \mathcal{P} \rightarrow \mathcal{Y}$. 
 Similarly to \cite{PetCocv11,Pet06}, we say that an LPV-SSA $\Sigma$ is a
 realization of $\Phi$, if for every $\mathfrak{F} \in \Phi$ there exists a state $x$ of $\Sigma$
 such that $\mathfrak{F}=\mathfrak{Y}_{\Sigma,x}$. Definition \ref{def:real} represents a 
 particular case of the definition above with $\Phi=\{\mathfrak{F}\}$.
The results of the paper can be extended to
 include the definition above, similarly to \cite{PetCocv11,Pet06}.
\end{remark}
Next, we define reachability and observability of LPV-SSAs.
\begin{definition}[Reachability \& observability] \label{def:reachobs}
  Let $\Sigma$ be an LPV-SSA representation of the form~\eqref{equ:alpvss}. We say that $\Sigma$ is \emph{(span) reachable} from an initial state $x_\mathrm{o} \in \X$, if  $\mathrm{Span}\{ \mathfrak{X}_{\Sigma,x_\mathrm{o}}(u,p)(t) \mid (u,p) \in \mathcal{U} \times \mathcal{P}, t \in \mathbb{T} \}=\mathbb{X}$. We say that $\Sigma$ is observable if, for any two states $x_1 \in \X$ and $x_2 \in \X$, $\mathfrak{Y}_{\Sigma,x_1} = \mathfrak{Y}_{\Sigma,x_2}$ implies $x_1 = x_2$.
\end{definition}
 Notice that, in this definition, observability means that for any two distinct states of the system, the resulting outputs will differ from each other for some input and scheduling signals. Notice that while span-reachability depends on the choice of the
 initial state $x_\mathrm{o}$, observability does not. Furthermore, these concepts of reachability and observability are strongly related to the extended controllability and observability matrices used in subspace-based identification of LPV-SSA models \cite{Win08}. 

As explained previously, the relation between two realizations of the same I-O function is of interest in this paper. Thus, it is essential to recall the notion of isomorphism for an LPV-SSA model.
  \begin{definition}[State-space isomorphism]
 \label{isomorphism}
   Consider two LPV-SSA representations $\Sigma=(\mathbb{P}, \left\{ A_i, B_i, C_i, D_i \right\}_{i=0}^{\QNUM})$  and $\Sigma^\prime=(\mathbb{P}, \left\{ A_i^{'}, B_i^{'}, C_i^{'}, D_i^{'} \right\}_{i=0}^{\QNUM})$ with $\dim(\Sigma)=\dim(\Sigma^\prime)=n_\mathrm{x}$. A nonsingular matrix $T \in \mathbb{R}^{n_\mathrm{x} \times n_\mathrm{x}}$ is said to be an isomorphism from $\Sigma$ to $\Sigma^\prime$, if
\begin{align}
 \label{equ:isomorphism}
  A^\prime_{i} T &= T A_i & B^\prime_{i} &= T B_i & C^\prime_{i}T &= C_i & D^\prime_i &=D_i,
\end{align}
for all $i\in\mathbb{I}_0^{n_\mathrm{p}}$. 
  \end{definition}
Next, we define minimality of an LPV-SSA representation: 
\begin{definition}[State-minimal realization]
\label{def:min}
 Let $\mathfrak{F}$ be an input-output function. An LPV-SSA $\Sigma$ is a (state) minimal realization of $\mathfrak{F}$, if
 \begin{itemize}
 \item $\exists x_\mathrm{o}\in\mathbb{X}$ such that $\mathfrak{Y}_{\Sigma,x_\mathrm{o}}= \mathfrak{F}$.
\item for ever LPV-SSA representation $\Sigma^{\prime}$  which is a realization of $\mathfrak{F}$, $\dim{ ( \Sigma ) } \leq \dim{ ( \Sigma^\prime ) }$.
 \end{itemize}
 We say that $\Sigma$ is \emph{minimal} w.r.t. the initial state $x_\mathrm{o}\in\mathbb{X}$,  if $\Sigma$ is a minimal realization of the input-output function $\mathfrak{Y}_{\Sigma,x_\mathrm{o}}$. 
\end{definition}
 
 Note that due to the linearity of the system class, we can assume that $D(\cdot)\equiv 0$ without any loss of generality regarding the concepts of reachability, observability and minimality.  Therefore, in the sequel, unless stated otherwise, we will assume that
 $D_i=0$ for all $i\in\mathbb{I}_0^{n_\mathrm{p}}$. Rewriting the results of the paper
 for $D(\cdot)\not\equiv 0$ is an easy exercise and it is left to
 the reader.
%

\section{Main results}
\label{para:main}
 In this section, we present the main results of the paper. First, we formally define the notion of an \emph{impulse response representation} (IIR) of
 an input-output function $\mathfrak{F}:\mathcal{U} \times \mathcal{P} \rightarrow \mathcal{Y}$ both in CT and DT. We then show  that all input-output functions which are realizable as a LPV-SSA representation admit such an IIR.
 This is followed by the establishment of a  Kalman-like realization theory (relationship between input-output functions and LPV-SSA representations, rank conditions for the Hankel matrix,  minimality of LPV-SSA representations, uniqueness (up to isomorphism) of minimal LPV-SSA representations). 
 Finally, we present a minimization and Kalman-decomposition algorithms, we discuss the correctness  of Kalman-Ho algorithm of \cite{TAW12}, and we conclude by clarifying 
 the relationship between the minimality concepts of the current paper and that of \cite{Tot10}. 

\subsection{Impulse response representation}
\label{para:io}

First, we introduce a convolution based representation of an input-output function. Let $p_q$ denote the $q^\mathrm{th}$ entry of the vector $p\in\mathbb{R}^{n_\mathrm{p}}$ if $q\in\mathbb{I}_1^{n_\mathrm{p}}$ and let $p_0=1$. Consider the following notation to handle the resulting $p$-dependence of the Markov coefficients:

%
\begin{definition}
\label{def:volterra}
 For a given index sequence $s \in \mathcal{S}(\mathbb{I}_{0}^{n_\mathrm{p}})$, time moments $t,\tau \in \mathbb{T}$, $\tau \le t$, and any scheduling trajectories  $p \in \mathcal{P}$, define 
 the so-called \emph{sub-Markov dependence} $(w_{s}\diamond p)(t,\tau)$ as follows:
 \begin{itemize}
 \item{\textbf{Continuous-time: }} For the empty sequence, $s=\epsilon$, $(w_{\epsilon}\diamond p)(t,\tau)=1$.  If $s=s_1s_2\cdots s_n$ for some $s_1,s_2,\ldots,s_n\in\mathbb{I}_{0}^{n_\mathrm{p}}$ and $n > 0$, 
then \vspace{-3mm}\begin{subequations}
 \begin{equation*}
  \begin{split}
    & (w_{s}\diamond p)(t,\tau)= \int_{\tau}^{t} p_{s_n}(\delta ) \cdot (w_{s_1\cdots s_{n-1}}\diamond p)(\delta,\tau)\ d\delta= \\
    & \int_{\tau}^t p_{s_n}(\tau_n) (\int_{\tau}^{\tau_n} p_{s_{n-1}}(\tau_{n-1}) (\int_{\tau}^{\tau_{n-1}} \cdots )d\tau_{n-1})d\tau_n   
 \end{split}
 \end{equation*}
 \item{\textbf{Discrete-time:}} If the sequence $s$ is of the form $s=s_1s_2 \cdots s_n$, for some $s_1,s_2,\ldots, s_n \in\mathbb{I}_{0}^{n_\mathrm{p}}$ and  $n=t-\tau+1$,  then 
 \begin{equation*}(w_{s}\diamond p)(t,\tau) =p_{s_1}(\tau)p_{s_2}(\tau+1) \cdots p_{s_n}(t) ,\end{equation*}\end{subequations}
 else $(w_{s}\diamond p)(t,\tau)=0$.
%
%
%
 \end{itemize}
 \end{definition}
\begin{example}
  In order to illustrate the notation above, consider the case when $\QNUM=1$ and take $s=0101$, $|s|=n=4$.  Then, for DT
  \( 
      (w_{s}\diamond p)(5,2)=p_0(2)p_1(3)p_0(4)p_1(5)=p(3)p(5)
  \).
  For CT,
  \(
     (w_{s}\diamond p)(5,2)= \int_2^{5} p_1(s_1) (\int_2^{s_1} p_0(s_2)(\int_2^{s_2} p_1(s_3) (\int_2^{s_3} p_0(s_4)ds_4)ds_3)ds_2)ds_1 \), 
  and by using $p_0=1$,  \( (w_s \diamond p)(5,2)=\int_2^5 p(s_1) (\int_2^{s_1} \int_2^{s_2} (s_3-2)p(s_3)ds_2ds_3)ds_1.\)
\end{example}
The IIR of an input-output function is defined as follows.
 \begin{definition}[Impulse response representation]
\label{def:grimarkov}
  Let $\mathfrak{F}$ be a function of the form~\eqref{equ:iofunction}. Then $\mathfrak{F}$  is said to have a \emph{impulse response representation} (IIR)  if there exists a function
  \begin{equation} \label{eq:subM}
    \theta_{\mathfrak{F}}: \mathcal{S}(\mathbb{I}_0^{n_\mathrm{p}}) \mapsto \mathbb{R}^{(\QNUM+1) n_\mathrm{y} \times (n_\mathrm{u}(\QNUM+1)+1)},
  \end{equation} 
  such that,
 \begin{enumerate}
 \item it satisfies an exponential growth condition, i.e., 
  there exist constants $K, R > 0$ such that 
  \begin{equation}
  \label{IIR:growth}
  \forall s \in   \mathcal{S}(\mathbb{I}_0^{n_\mathrm{p}}):  ||\theta_{\mathfrak{F}}(s)||_\mathrm{F} \le K R^{|s|}
  \end{equation}
  where $\|.\|_\mathrm{F}$ denotes the Frobenius norm;
 \item for every $p \in \mathcal{P}$, 
 there exist functions $g_{\mathfrak{F}} \diamond p:\mathbb{T} \rightarrow \mathbb{R}^{n_\mathrm{y}}$ and
 $h_{\mathfrak{F}} \diamond p:\{(\tau,t) \in \mathbb{T} \times \mathbb{T} \mid \tau \le t\} \rightarrow \mathbb{R}^{n_\mathrm{y} \times n_\mathrm{u}}$, such that
 for each $(u,p) \in \mathcal{U} \times \mathcal{P}$, $t \in \mathbb{T}$, \vspace{-2mm}
 \begin{subequations}\label{equ:convol}
  \begin{equation}
  \mathfrak{F}(u,p)(t)=(g_{\mathfrak{F}} \diamond p)(t) + \int_0^{t} (h_{\mathfrak{F}} \diamond p)(\delta,t)\cdot u(\delta)\ d\delta,
   \end{equation}
  in CT and
   \begin{equation}
  \mathfrak{F}(u,p)(t)=(g_{\mathfrak{F}} \diamond p)(t) + \sum_{\delta=0}^{t-1} (h_{\mathfrak{F}} \diamond p)(\delta, t)\cdot u(\delta),
%
%
 \end{equation}
 \end{subequations}
  for DT. Moreover, 
  for any $i,j \in \mathbb{I}_0^{n_\mathrm{p}}$, let $\eta_{i,\mathfrak{F}}(s) \in \mathbb{R}^{n_\mathrm{y} \times 1}$ and
 $\theta_{i,j,\mathfrak{F}}(s) \in \mathbb{R}^{n_\mathrm{y} \times n_\mathrm{u}}$ be such that
 \begin{equation*} 
\theta_{\mathfrak{F}}(s)=\begin{bmatrix}
            \eta_{0,\mathfrak{F}}(s) & \theta_{0,0,\mathfrak{F}}(s) & \cdots & \theta_{0,\QNUM,\mathfrak{F}}(s) \\ 
            \eta_{1,\mathfrak{F}}(s) & \theta_{1,0,\mathfrak{F}}(s) & \cdots & \theta_{1,\QNUM,\mathfrak{F}}(s) \\ 
            \vdots       & \vdots       & \cdots & \vdots  \\ 
            \eta_{\QNUM,\mathfrak{F}}(s) & \theta_{\QNUM,0,\mathfrak{F}}(s) & \cdots & \theta_{\QNUM,\QNUM,\mathfrak{F}}(s) \\ 
           \end{bmatrix}.
  \end{equation*}
Then $g_{\mathfrak{F}} \diamond p$ and $h_{\mathfrak{F}} \diamond p$ can be expressed via $\theta_{\mathfrak{F}}$ as
 \begin{equation}
  \label{IIRsum:CT}
   \begin{split}
    & (g_{\mathfrak{F}} \diamond p)(t)= \\
    & \sum_{i \in \mathbb{I}_{0}^{n_\mathrm{p}}} \sum_{ s \in \mathcal{S}(\mathbb{I}_{0}^{n_\mathrm{p}})}\!\!\! p_i(t) \eta_{i,\mathfrak{F}}(s)\cdot  (w_s\diamond p)(t,0),\\
   & (h_{\mathfrak{F}}\diamond p)(\delta,t)=  \\
   & \sum_{i,j \in \mathbb{I}_{0}^{n_\mathrm{p}}} \sum_{s \in \mathcal{S}(\mathbb{I}_{0}^{n_\mathrm{p}})}\!\!\! \theta_{i,j,\mathfrak{F}}(s) p_{i}(t)p_{j}(\delta) \cdot (w_{s}\diamond p)(t,\delta), 
 \end{split}
 \end{equation}
 in CT and, in DT,
  \begin{equation}
  \label{IIRsum:DT}
   \begin{split}
   & (g_{\mathfrak{F}}\diamond p)(t)=\\
   &  \sum_{\substack{ i \in \mathbb{I}_{0}^{n_\mathrm{p}}\\ s \in \mathcal{S}(\mathbb{I}_{0}^{n_\mathrm{p}})} }\!\!\! \eta_{i,\mathfrak{F}}(s)p_{i}(t)\cdot  (w_s\diamond p)(t-1,0),\\
   & (h_{\mathfrak{F}} \diamond p)(\delta,t)=\\
   & \sum_{\substack{ i,j \in \mathbb{I}_{0}^{n_\mathrm{p}}\\ s \in \mathcal{S}(\mathbb{I}_{0}^{n_\mathrm{p}}) }}\!\!\! \theta_{i,j,\mathfrak{F}}(s) p_{i}(t)p_{j}(\delta) \cdot (w_{s}\diamond p)(t-1,\delta+1). 
 \end{split}
 \end{equation}
 \end{enumerate}
 If $\mathfrak{F}$ is clear from the context, then we drop the subscript $\mathfrak{F}$ and we denote $\theta_{\mathfrak{F}}, \theta_{i,j,\mathfrak{F}}, \eta_{i,\mathfrak{F}}, i,j \in \AQ,  g_{\mathfrak{F}}\diamond p, h_{\mathfrak{F}}\diamond p$ by $\theta, \theta_{i,j},\eta_i$ $g \diamond p$ and $h \diamond p$, respectively. 
 The values of the function $\theta_{\mathfrak{F}}$ will be called the \emph{sub-Markov parameters} of  $\mathfrak{F}$.
 \end{definition}\vskip 4mm
  \vspace{-10pt}
Note that in DT, the sums appearing on the right-hand side of  \eqref{IIRsum:DT} are actually finite sums, as 
for $|s| > t$, $w_s \diamond p = 0$. 
In the case of CT, however, the right-hand side of \eqref{IIRsum:CT} is an infinite sum, which raises the question of its convergence.
This question is addressed below.
\begin{lemma}
\label{lemma:IIRcon}
 Assume $\theta_{\mathfrak{F}}$ satisfies the growth condition 
 \eqref{IIR:growth} for some $K,R > 0$. Then the infinite sum on the right-hand side of \eqref{IIRsum:CT} is absolutely convergent.
\end{lemma}
The proof of Lemma \ref{lemma:IIRcon} is presented in Appendix.
Existence of an IIR of $\mathfrak{F}$ implies that $\mathfrak{F}$ is linear in $u$ and can be represented as a convergent
infinite sum of iterated integrals in CT, while, in DT, $\mathfrak{F}$  is a homogeneous polynomial in $\{p_i(t)\}_{i=1}^{n_\mathrm{p}}$.
It is important to  notice that, in CT, using the terminology of
\cite{Isi:Nonlin}, the entries of $g \diamond p$ and $h \diamond p $ correspond to
the generating series defined by the coordinates of 
the functions $s \mapsto \theta_{i,j}(s)p_i(t)p_{j}(\tau)$ and
$s \mapsto \eta_{i}(s)p_{i}(t)$. Furthermore, the above definition of IIRs, in principle, corresponds to a specific case of IIR for general LPV systems defined for the DT case in \cite{Toth2010SpringerBook}. Note that in \cite{Toth2010SpringerBook} the use of an initial condition was avoided by assuming that the input-output function is asymptotically stable. The contribution in the definition proposed in the current paper is twofold: (i) it provides the concept of IIR for the CT case, (ii) it restricts the dependencies of the Markov parameters to the subclass that can results from the series expansion of LPV-SSA representations. As we will see, this will be crucial to decide when it is possible to derive a LPV-SSA realization of an input-output function. In turn, that result provides the basis for system identification with state-space model structures using static dependence only.
\begin{example}
To better explain the meaning of this definition, we demonstrate the underlying constructive mechanism by writing out the formulas explicitly for a single example. Assume that $\mathbb{P}=\mathbb{R}$, 
$\NU=\NY=1$ and let $\mathfrak{F}$ be an input-output function of the form \eqref{equ:iofunction} and assume it has an IIR. Then in DT, using that $p_0(t)=1$ for all $t \in \mathbb{T}$,
\vspace{-10pt}
\[ 
   \begin{split}
    & (h_{\mathfrak{F}} \diamond p)(2,5)=  \\
   & \theta_{0,0,\mathfrak{F}}(00)+ p(4)\theta_{0,0,\mathfrak{F}}(01)+  \cdots + \\
  & p(2)p(5)p(3)\theta_{1,1,\mathfrak{F}}(10)+ p(2)p(5)p(3)p(4)\theta_{1,1,\mathfrak{F}}(11) \\
    & (g_{\mathfrak{F}} \diamond p)(2)=    \eta_{0,\mathfrak{F}}(00)+    p_1(1)\eta_{0,\mathfrak{F}}(01)+     p_1(0)\eta_{0,\mathfrak{F}}(10)+  \\
    & p_1(0)p_1(1)\eta_{0,\mathfrak{F}}(11)+ \cdots + p_1(2)\eta_{1\mathfrak{F}}(00)+  \\
    &p_1(2)p_1(1)\eta_{1,\mathfrak{F}}(01)+ \cdots p_1(2)p_1(0)p_1(1)\eta_{1,\mathfrak{F}}(11). \\
  \end{split}
 \]             
\vspace{-10pt}
  For CT,
  \[
   \begin{split}
    & (h_{\mathfrak{F}} \diamond p)(2,5)=  
       [\theta_{0,0,\mathfrak{F}}(\epsilon) + 3\theta_{0,0,\mathfrak{F}}(0) + \\
        & +\cdots + \theta_{0,0,\mathfrak{F}}(101)\int_2^5 p(s_1)\int_2^{s_1} \int_2^{s_2} p(s_3)ds_3 ds_2ds_1+\cdots ] \\
      &  +  \cdots + \\
      & p(2)p(5)[\theta_{1,1,\mathfrak{F}}(\epsilon) + 3\theta_{1,1,\mathfrak{F}}(0) + \theta_{1,1,\mathfrak{F}}(1)\int_2^{5} p(s)ds + \\
        &  + \cdots + \theta_{1,1,\mathfrak{F}}(101)\int_2^5 p(s_1)\int_2^{s_1} \int_2^{s_2} p(s_3)ds_3 ds_2ds_1+\cdots ] \\
   \end{split}
  \]
  \[ \begin{split}
   & (g_{\mathfrak{F}} \diamond p)(2) = [\eta_{0,\mathfrak{F}}(\epsilon) + 2\eta_{0,\mathfrak{F}}(0) + \eta_{0,\mathfrak{F}}(1)\int_0^{2} p(s)ds + \\
        & +\cdots + \eta_{0,\mathfrak{F}}(101)\int_0^2 p(s_1)\int_0^{s_1} \int_0^{s_2} p(s_3)ds_3 ds_2ds_1+\cdots ] \\
     & + 
      p(2)[\eta_{1,\mathfrak{F}}(\epsilon) + 2\eta_{1,\mathfrak{F}}(0) + \eta_{1,\mathfrak{F}}(1)\int_0^{2} p(s)ds  \\
        &  + \cdots + \eta_{1,\mathfrak{F}}(101)\int_0^2 p(s_1)\int_0^{s_1} \int_0^{s_2} p(s_3)ds_3 ds_2ds_1+\cdots ].  \\ 
  \end{split}
  \]
  That is, in DT,  $(h_{\mathfrak{F}} \diamond p)(2,5)$ is a polynomial of $p(2),p(3), p(4),p(5)$, and the degree of $p(2),p(3),p(4),p(5)$ in each monomial is at most one. Moreover, 
  $\theta_{i,j,\mathfrak{F}}(s_1s_2)$, for each $i,j,s_1, s_2 \in \{0,1\}$, are the coefficients of this polynomial. In particular, only the components of the sub-Markov parameters 
   the form $\theta_{\mathfrak{F}}(s)$, with $s$ being of length $2$, occur in  $(h_{\mathfrak{F}} \diamond p)(2,5)$.
   In contrast, in CT, $(h_{\mathfrak{F}} \diamond p)(2,5)$ is an infinite sum of iterated integrals of $p$, all the components of the form $\theta_{i,j,\mathfrak{F}}(s)$, $i,j=0,1$, with $s$ being a 
   sequence of arbitrary length,  occur in 
   the expression for  $(h_{\mathfrak{F}} \diamond p)(2,5)$. The picture for $(g_{\mathfrak{F}} \diamond p)(2)$ is analogous. 
\end{example}


Recall that for LTI systems, there is a one-to-one correspondence between input-output functions and IIRs (see, e.g., \cite{Kailath80}). A similar result holds for those  functions of the form \eqref{equ:iofunction} which admit an IIR.
\begin{lemma}[Uniqueness of the IIR]\label{lem:extension}
  If an input-output function $\mathfrak{F}$  of the form \eqref{equ:iofunction} has an IIR, then the function $\theta_{\mathfrak{F}}$ is uniquely determined by $\mathfrak{F}$, i.e., if $\hat{\mathfrak{F}}:\mathcal{U} \times \mathcal{P} \rightarrow \mathcal{Y}$
  is another input-output function, which admits an IIR, then 
  \[ \mathfrak{F}=\hat{\mathfrak{F}}\ \iff\ \theta_\mathfrak{F} = \theta_{\hat{\mathfrak{F}}} .\]
 Moreover, there exists a unique extension $\mathfrak{F}_\mathrm{e}$ of  $\mathfrak{F}$ to $\mathcal{U} \times \mathcal{P}_{\mathrm e}$, where 
 $\mathcal{P}_\mathrm{e}=\mathcal{C}_\mathrm{p}(\mathbb{R}_0^+,\mathbb{R}^{\QNUM})$ in CT or 
 $\mathcal{P}_\mathrm{e}=(\mathbb{R}^{\QNUM})^{\mathbb{N}}$ in DT.
 The extension $\mathfrak{F}_\mathrm{e}$ also admits an IIR and $\theta_\mathfrak{F}=\theta_{\mathfrak{F}_\mathrm{e}}$.
 \end{lemma}
  The proof of this result is given in the Appendix. This result not only yields a one-to-one correspondence between
  input-output maps and sub-Markov parameters, but it also tells us that the choice of scheduling space does not matter, since
  we can always extend an input-output function to a larger scheduling space or restrict it to a smaller one in a unique fashion.
  In particular, it will allow us to reduce realization theory of LPV-SSA representations to that of linear switched system, and use
  the results of \cite{Pet12,PetCocv11}.
 In the sequel, we will restrict our attention to input-output functions which admit an IIR in the previously defined form. This is not a serious restriction since any input-output function of a LPV-SSA representation always admits an IIR:
 \begin{lemma}[Existence of the IIR]\label{lem:realiofunction}
 The LPV-SSA representation $\Sigma$ of the form~\eqref{equ:alpvss} is a realization of an input-output function $\mathfrak{F}$ of the form \eqref{equ:iofunction}, if and only if, $\mathfrak{F}$ has an IIR and, for all $i,j\in \mathbb{I}_0^{n_\mathrm{p}}$, $s \in \Words$,  this IIR is such that \vspace{-1mm}
 \begin{subequations}\label{lem:realiofunction:eq0}
 \begin{align}
  \eta_{i,\mathfrak{F}}(s) &= C_i  A_s  x_\mathrm{o},\\
 \theta_{i,j,\mathfrak{F} }(s) &= C_i A_s B_{j} 
\end{align}
\end{subequations}
where for $s=\epsilon$, $A_s$ denotes the identity matrix, and for 
$s = s_1 \cdots s_{n-1}s_n $ and $s_1,\ldots s_n \in\mathbb{I}_0^{n_\mathrm{p}}$, $n > 0$, $A_{s}=A_{s_n}A_{s_{n-1}} \cdots A_{s_1}$. 
 \end{lemma} 
 The proof of this result is given in the Appendix. 
\begin{remark}[Further intuition behind IIR representation]
From the proof Lemma \ref{lem:realiofunction} it also follows that, 
if $\mathfrak{F}$ is realized by an LPV-SSA representation $\Sigma$  of the form \eqref{equ:alpvss} from the initial state $x_{\mathrm o}$,  
for all $\tau \le t \in \mathbb{T}$, $p \in \mathcal{P}$, 
\[ 
   \begin{split}
& (g_{\mathfrak{F}} \diamond p)(t)=\left\{\begin{array}{rl}]
                                    C(p(t))\Phi_{p}(t-1,0)x_{\mathrm o} & \mbox{ DT} \\
                                    C(p(t))\Phi_{p}(t,0)x_{\mathrm o} & \mbox{ CT} 
                                 \end{array}\right. \\
 & (h_{\mathfrak{F}} \diamond p)(\tau,t)=\left\{\begin{array}{rl}
                                      C(p(t))\Phi_{p}(t,\tau+1)B(p(\tau)) & \mbox{ DT } \\
                                      C(p(t))\Phi_{p}(t,\tau)B(p(\tau)) & \mbox{ CT } 
                                      \end{array}\right.
\end{split}
\]
Here $\Phi_{p}(t,\tau)$ is the fundamental matrix of
the time-varying linear system $\xi x(t)=A(p(t))x(t)$, i.e. $\Phi_p(\tau,\tau)=I_{\NX}$ and for all $\tau \le t \in \mathbb{T}$,
 $\dfrac{d}{dt} \Phi_p(t,\tau)=A(p(t))\Phi_p(t,\tau)$ in CT and $\Phi_p(t+1,\tau)=A(p(t))\Phi_p(t,\tau)$ in DT. 
\end{remark}

\subsection{State-space realization theory for affine dependence}
\label{para:kalman}
 Below, we present a novel Kalman-style realization theory for LPV-SSA representations, which, in our opinion, 
 opens the door for the development of a new generation of state-space identification, model reduction and control methodologies.  
\begin{theorem}[Minimality, weak sense]
\label{theo:min}
  An LPV-SSA representation  $\Sigma$ is minimal w.r.t. a given initial state $x_\mathrm{o}\in\mathbb{X}$, if and only if, $\Sigma$ is 
span-reachable from $x_\mathrm{o}$ and $\Sigma$ is observable. If $\Sigma$ is an LPV-SSA representation which is minimal w.r.t. some initial state $x_0$, and $\Sigma^\prime$ is an LPV-SSA representation which is minimal w.r.t. some initial state $x_0^{'}$, 
and $\Sigma$ and $\Sigma^\prime$ are weakly input-output equivalent w.r.t the initial states $x_0$ and $x_0^\prime$, then $\Sigma$ and $\Sigma^\prime$ are isomorphic. 
\footnote{In fact, with the notation of Definition \ref{isomorphism}, we can show that there exists a matrix $T$ such that in addition to \eqref{equ:isomorphism}, $Tx_0=x^{'}_0$ holds. See the discussion after the proof of Theorem \ref{theo:min} in Appendix.}
\end{theorem}
The proof of this result is given in the Appendix. Another, equivalent way to state Theorem \ref{theo:min} is as follows:
\begin{theorem}[Minimal realizations, alternative statement]
 Assume $\mathfrak{F}$ is an input-output map of the form \eqref{equ:iofunction}. If an LPV-SSA $\Sigma$ is a realization of $\mathfrak{F}$ from the initial
 state $x_{\mathrm o}$, then $\Sigma$ is a minimal realization of $\mathfrak{F}$ if and only if $\Sigma$ is observable and span-reachable from $x_{\mathrm o}$. Any two minimal LPV-SSA realizations of $\mathfrak{F}$ are isomorphic. 
\end{theorem}
If we restrict our attention to the case of zero initial state, then Theorem \ref{theo:min} can be restated as follows: an LPV-SSA representation is minimal w.r.t. zero initial state, if and
only if it is observable and span-reachable from zero. Any two LPV-SSA representations which are minimal and weakly input-output equivalent w.r.t. the zero initial state (i.e., which induce the same input-output function from the zero initial state and which are both minimal realizations of this input-output function from zero), are isomorphic. 
Another consequence of Theorem \ref{theo:min} is that weak input-output equivalence of two LPV-SSA representations with respect to some initial states implies strong input-output equivalence of these representations, provided that both representations are minimal w.r.t. the designated initial states. This follows by noticing that these LPV-SSA representations are isomorphic, and hence they are strongly equivalent.
This opens up the possibility to deal with strong minimality.
Let us call an LPV-SSA $\Sigma$ \emph{strongly minimal}, if $\Sigma$ is minimal w.r.t. all $x_\mathrm{o}\in\mathbb{X}$.
\begin{theorem}[Minimality, strong sense]
\label{theo:min:strong}
 An LPV-SSA representation $\Sigma$ is strongly minimal, $\iff$ it is minimal w.r.t. $0$ $\iff$ it is observable and span-reachable from the zero initial state. 
Any two strongly minimal and strongly input-output equivalent LPV-SSA representations are isomorphic. 
In addition, two strongly minimal LPV-SSA representations are weakly input-output equivalent w.r.t. to some initial states if and only if they are strongly input-output equivalent. 
\end{theorem}
The proof of Theorem \ref{theo:min:strong} is presented in the Appendix. Theorem \ref{theo:min:strong} implies that LPV-SSA representations which are minimal w.r.t. the zero initial state have particularly nice properties. Note that
it is perfectly possible for an LPV-SSA representation to be minimal w.r.t. some initial state, and not to be minimal w.r.t. the zero initial state.

A remarkable observation is that, similarly to the linear time-invariant case, 
rank conditions for observability and reachability can be obtained to verify state minimality for LPV-SSA, which is not the case for general LPV-SS representations (see \cite{Toth11_LPVBehav}). To this end, let us recall the definition of the extended reachability and observability matrices for LPV-SSA representations (see, \emph{e.g.}, \cite{TAW12}).
Let $\Sigma$ be an LPV-SSA representation of the form~\eqref{equ:alpvss}-\eqref{equ:affdep} with $D(\cdot)\equiv 0$.
\begin{definition}[Ext. reachability \& observability matrices]
For an initial state $x_{\mathrm{o}}$, the $n$-step extended reachability matrices
$\mathcal{R}_n$  of $\Sigma$ from $x_{\mathrm{o}}$, $n \in \mathbb{N}$, are defined recursively as follows
\begin{subequations}
  \begin{align}
    \mathcal{R}_0 &=\left[
    \begin{array}{cccc}
      x_\mathrm{o} & B_0  & \cdots & B_{n_p}
    \end{array}\right], \\
    \mathcal{R}_{n+1} &=
    \left[
    \begin{array}{cccc} \mathcal{R}_n & A_0 \mathcal{R}_n & \cdots &
      A_{n_\mathrm{p}} \mathcal{R}_n \end{array} \right],
  \end{align}
\end{subequations}
The extended $n$-step observability matrices
$\mathcal{O}_n$ of $\Sigma$, $n \in \mathbb{N}$, are given as
\begin{subequations}
  \begin{align}
    \mathcal{O}_0 &=     \left[\begin{array}{cccc}  C_0^\top &  \cdots & C_{n_\mathrm{p}}^\top \end{array}\right]^\top, \\
    \mathcal{O}_{n+1} &=
    \left[\begin{array}{cccc}  
      \mathcal{O}_n^\top & 
      A_0^\top \mathcal{O}_{n}^\top  &
      \cdots & A_{n_\mathrm{p}}^\top \mathcal{O}^\top_n \end{array}\right]^\top.
  \end{align}
\end{subequations}
\end{definition}
It is not difficult to show that \begin{subequations} \begin{equation} \mathrm{Im} \{ \mathcal{R}_{n_\mathrm{x}-1}\} =\sum_{i=0}^{\infty} \mathrm{Im}\{ \mathcal{R}_{i} \},  \vspace{-1mm}\end{equation} and $\mathcal{R}_\ast:=\mathrm{Im} \{ \mathcal{R}_{n_\mathrm{x}-1}\}$ is the smallest subspace of $\mathbb{R}^{n_\mathrm{x}}$ 
such that
$x_\mathrm{o} \in \mathcal{R}_\ast$,  $\mathrm{Im} B_i \subseteq \mathcal{R}_\ast$, $i \in \mathbb{I}_0^{\QNUM}$ and invariant in the sense: $A_i\mathcal{R}_\ast \subseteq \mathcal{R}_\ast$,
$\forall i \in \mathbb{I}_0^{n_\mathrm{p}}$. 
Similarly, \begin{equation} \mathrm{Ker} \{ \mathcal{O}_{n_\mathrm{x}-1}\} = \bigcap_{i=0}^{\infty} \mathrm{Ker} \{ \mathcal{O}_i \},\end{equation}\end{subequations}  and hence
$\mathcal{O}_\ast := \mathrm{Ker} \{ \mathcal{O}_{n_\mathrm{x}-1}\} $ is the largest subspace of $\mathbb{R}^{n_\mathrm{x}}$
such that $\mathcal{O}_\ast  \subseteq \mathrm{Ker} \{C_i\}$ and 
$A_i \mathcal{O}_\ast \subseteq \mathcal{O}_\ast $, $\forall i \in \mathbb{I}_0^{n_\mathrm{p}}$. %
Note that while the extended reachability matrices are defined from a particular initial state, the extended observability matrices do not depend on the choice of the initial state. 

\begin{theorem}[Rank conditions]
\label{theo:reachobs:rank:cond}
 The LPV-SSA representation $\Sigma$ is span-reachable from $x_\mathrm{o}$, if and only if $\rank{ \left( \mathcal{R}_{n_\mathrm{x}-1} \right) } = n_\mathrm{x}$.
 $\Sigma$ is observable, if and only if $\rank{ \left( \mathcal{O}_{n_\mathrm{x}-1} \right) } = n_\mathrm{x}$.
\end{theorem}
The proof is given in the Appendix.  
This Theorem directly leads to algorithms for reachability, observability and minimality reduction of LPV-SSA models. These algorithms are similar as those for linear switched systems (see, \emph{e.g.}, \cite{Pet12,PetCocv11}).
\begin{Procedure}[Reachability reduction]
\label{LSSreach} Let $\Rank (\mathcal{R}_{n_\mathrm{x}-1}) = r$ and choose a basis $\{b_i\}_{i=1}^{n_\mathrm{x}} \subset \mathbb{R}^{n_\mathrm{x}}$ such that $\mathrm{Span}\{b_1,\ldots,b_{r} \}= \mathrm{Im} \{ \mathcal{R}_{n_\mathrm{x}-1}\} $.  In the new basis, the matrices
$\{A_i,B_i,C_i\}_{i=0}^{n_\mathrm{p}}$ become
\begin{subequations}\label{LSSreach:eq1} 
\begin{align}
\hat A_{i}&=\begin{bmatrix} A_{i}^{\mathrm R} & A^\prime_{i} \\ 0 & A^{\prime\prime}_{i} \end{bmatrix},&
\hat B_{i}&=\begin{bmatrix} B_{i}^{\mathrm R} \\ 0 \end{bmatrix}, \\
 \hat C_{i}&=\begin{bmatrix} C_i^{\mathrm R} & C_{i}^{\prime} \end{bmatrix}, & \hat x_\mathrm{o}&=\begin{bmatrix} x_\mathrm{o}^{\mathrm R} \\ 0 \end{bmatrix},
\end{align} \end{subequations}
where $A^{\mathrm R}_i \in \mathbb{R}^{r \times r}, B_i^{\mathrm R}
\in \mathbb{R}^{r \times n_\mathrm{u}}$, and $C^{\mathrm R}_i \in \mathbb{R}^{n_\mathrm{y} \times r}$. 
Define $\Sigma^{\mathrm R}= (\mathbb{P},\{A_i^{\mathrm R}, B_i^{\mathrm R}, C_i^{\mathrm R}\}_{i=0}^{n_\mathrm{p}})$. 
Then $\Sigma^{\mathrm R}$
is span-reachable from $x^{\mathrm{R}}_0$ and $\Sigma$ and $\Sigma^{\mathrm R}$
 are weakly input-output equivalent w.r.t. $x_{\mathrm o}$ and $x_{\mathrm o}^{\mathrm{R}}$, i.e. $\mathfrak{Y}_{\Sigma,x_{\mathrm o}}=\mathfrak{Y}_{\Sigma^{\mathrm R},x_{\mathrm o}^{\mathrm R}}$.
\end{Procedure}
Intuitively, $\Sigma^{\mathrm R}$ is obtained from $\Sigma$ by
restricting the dynamics and the output function of $\Sigma$ to the
subspace $\mathrm{Im}\{ \mathcal{R}_{n_\mathrm{x}-1}\}$. 
\begin{Procedure}[Observability reduction]
\label{LSSobs}
Let $\rank (\mathcal{O}_{n_\mathrm{x}-1})=o$ and choose a basis $\{b_i\}_{i=1}^{n_\mathrm{x}} \subset \mathbb{R}^{n_\mathrm{x}}$ such that  $\mathrm{Span}\{b_{o+1},\ldots,b_{n_\mathrm{x}} \}= \mathrm{Ker} \{ \mathcal{O}_{n_\mathrm{x}-1}\} $.  In the new basis, the matrices
$\{A_i,B_i,C_i\}_{i=0}^{n_\mathrm{p}}$ become
\begin{subequations}\label{LSSobsv:eq1} 
\begin{align}
\hat A_{i}&=\begin{bmatrix} A_{i}^{\mathrm O} & 0 \\ A^\prime_{i} & A^{\prime\prime}_{i} \end{bmatrix},&
\hat B_{i}&=\begin{bmatrix}  B_{i}^{\mathrm O} \\ B_i^{\prime} \end{bmatrix}, \\
 \hat C_{i}&=\begin{bmatrix} C_i^{\mathrm O} & 0 \end{bmatrix}, & \hat{x}_\mathrm{o}&=\begin{bmatrix} x_{\mathrm{o}}^{\mathrm{O}} 
 \\ x_\mathrm{o}^{\prime} 
 \end{bmatrix},
\end{align} \end{subequations}
where $A^{\mathrm O}_{i} \in \mathbb{R}^{o \times o}, B_i^{\mathrm O}
\in \mathbb{R}^{o \times n_\mathrm{u}}$ and $C_i^{\mathrm O} \in \mathbb{R}^{n_\mathrm{y}
  \times o}$.  Define $\Sigma^{\mathrm O}= (\mathbb{P},\{A_i^{\mathrm O}, B_i^{\mathrm O}, C_i^{\mathrm O}\}_{i=0}^{n_\mathrm{p}})$. Then,  any $x^{\mathrm{O}}_\mathrm{o}\in\mathbb{R}^{o}$ is observable, and $\Sigma$ and $\Sigma^{\mathrm O}$
 are weakly input-output equivalent w.r.t. $x_{\mathrm o}$ and $x_{\mathrm o}^{\mathrm{O}}$, i.e. $\mathfrak{Y}_{\Sigma,x_{\mathrm o}}=\mathfrak{Y}_{\Sigma^{\mathrm O},x_{\mathrm o}^{\mathrm O}}$.
%
%
%
%
\end{Procedure}
Intuitively, $\Sigma^{\mathrm O}$ is obtained from $\Sigma$ by merging
any two states $x_1$, $x_2$ of $\Sigma$, for which
$\mathcal{O}_{n_\mathrm{x}-1}x_1=\mathcal{O}_{n_\mathrm{x}-1}x_2$.

\begin{Procedure}[Minimal representation]
\label{LSSmin}
Given an LPV-SSA representation $\Sigma$ and an initial state $x_\mathrm{o}\in\mathbb{R}^{n_\mathrm{x}}$. Using Procedure \ref{LSSreach}, transform $\Sigma$
 w.r.t. $x_\mathrm{o}$ to a span reachable $\Sigma^{\mathrm R}$. Subsequently, transform $\Sigma^{\mathrm R}$ w.r.t. $x_\mathrm{o}^{\mathrm R}$
to an observable $\Sigma^{\mathrm M}$ with $x_\mathrm{o}^{\mathrm M}$ using Procedure~\ref{LSSobs}.  
Then,
$\Sigma^{\mathrm M}$ is a minimal LPV-SSA w.r.t. $x_\mathrm{o}^{\mathrm M}$ and $\Sigma^{\mathrm M}$ is weakly input-output equivalent to
$\Sigma$ w.r.t. initial states $x_{\mathrm o}^{\mathrm M}$ and $x_{\mathrm o}$.
\end{Procedure}

Procedures \ref{LSSreach} -- \ref{LSSobs} can be combined to yield a Kalman-decomposition as follows.
\begin{Procedure}[Kalman decomposition]
\label{kalman_decomp}
Consider an LPV-SSA $\Sigma$ of the form \eqref{equ:alpvss} and an initial state $x_0 \in \X$. 
 Choose a basis $\{b_i\}_{i=1}^{n_\mathrm{x}} \subset \mathbb{R}^{n_\mathrm{x}}$ such that $\mathrm{Span}\{b_1,\ldots,b_{r} \}\!\!\!=\!\!\! \mathrm{Im} \{ \mathcal{R}_{n_\mathrm{x}-1}\} $ and
$\mathrm{Span}\{b_{r_m+1},\ldots,b_{r} \}\!\!\!=\!\!\!(\mathrm{Im}\{ \mathcal{R}_{n_\mathrm{x}-1}\} \cap \ker \{\mathcal{O}_{n_{\mathrm x}-1}\})$ for some 
$r,r_m \ge  0$. Define $T=\begin{bmatrix} b_1 & b_2 & \ldots  & b_{n_{\mathrm x}} \end{bmatrix}^{-1}$, and let
$\hat{A}_i=TA_iT^{-1}$, $\hat{B}_i=TB_i, \hat{C}_i=C_iT^{-1}$, $i \in \AQ$, $\hat{x}_{\mathrm o}=Tx_{\mathrm o}$. Then 
\begin{equation}
\label{kalman_decomp}
\begin{split}
& \hat{A}_{i} \!\!=\!\!\begin{bmatrix} A_{i}^{\mathrm m} & \!\!\! 0 & \!\!\! A^{\prime\prime}_i \\
                                \!\!\! A^{\prime}_i    & \hat{A}^{\prime}  &\!\!\! A^{\prime\prime\prime}_i  \\
                            0                 & \!\!\! 0            &\!\!\! A^{\prime\prime\prime\prime}_i \!\!\!
             \end{bmatrix}\!\!, ~
 \hat{B}_{i}\!\!=\!\!\begin{bmatrix} B_{i}^{\mathrm m} \\ B_{i}^{\prime} \\  0 \!\! \end{bmatrix}\!\!, ~  
\hat{C}_{i} \!\!=\!\!\begin{bmatrix} (C_i^{\mathrm m})^{\top} \\  0 \\ (C_{i}^{\prime})^{\top} \!\! \end{bmatrix}^{\top}\!\!\!, \\
 & \hat{x}_\mathrm{o}\!\!=\!\!\begin{bmatrix} (x_\mathrm{o}^{\mathrm m})^{\top} & \bar{x}_\mathrm{o}^{\top} & 0 \end{bmatrix}^{\top},
\end{split}
\end{equation}
where $A^{\mathrm m}_i \in \mathbb{R}^{r_m \times r_m}, B_i^{\mathrm m} \in \mathbb{R}^{r_m \times n_\mathrm{u}}$, and $C^{\mathrm m}_i \in \mathbb{R}^{n_\mathrm{y} \times r_m}$,
$x_{\mathrm{o}}^m \in \mathbb{R}^{r_m}$, and  $A_i^{\prime\prime\prime} \in \mathbb{R}^{(n-r) \times (n-r)}$, $A_i^{\prime} \in \mathbb{R}^{(r-r_m) \times r_m}$, $A^{\prime\prime}_i  \in \mathbb{R}^{r_m \times (n-r)}$, $ A^{\prime\prime\prime}_i  \in \mathbb{R}^{(r-r_m) \times (n-r)}$, $\hat{A}^{\prime} \in \mathbb{R}^{(r-r_m) \times (r-r_m)}$, $B_i^{\prime} \in \mathbb{R}^{(r-r_m) \times \NU}$, $C_i^{\prime} \in \mathbb{R}^{\NY \times (n-r)}$.
Clearly, $\hat{\Sigma}=(\mathbb{P}, \{\hat{A}_i, \hat{B}_i, \hat{C} _i,0\}_{i=0}^{n_\mathrm{p}})$ is isomorphic to $\Sigma$ and can be viewed as its Kalman-decomposition of $\Sigma$.
\end{Procedure}
\begin{corollary}
\label{lem:min}
 The LPV-SSA
$\Sigma^{\mathrm m}= (\mathbb{P},\{A_i^{\mathrm m}, B_i^{\mathrm m}, C_i^{\mathrm m},0\}_{i=0}^{n_\mathrm{p}})$ is a minimal realization of $\mathfrak{F}=\mathfrak{Y}_{\Sigma,x_0}$ 
from the intial state  $x^{\mathrm m}_{\mathrm o}$.
\end{corollary}
The proof of Corollary \ref{lem:min} is presented in Appendix.

In order to demonstrate what the corresponding span-reachable and observable representations really describe let fix the scheduling trajectory $p\in\mathcal{P}$. Then, the LPV-SSA representation $\Sigma$ is equivalent with a a \emph{linear time-varying} (LTV)
representation
\begin{subequations}
\begin{align}
 \xi x(t)&=A(t)x(t)+B(t)u(t), \\
 y(t) &= C(t)x(t)+D(t)u(t), 
\end{align}
\end{subequations}
where $A(t):=A(p(t)),\ldots,D(t):=D(p(t))$. Let us introduce the following definitions:
%
\begin{definition}[Regularity certificate]
 Let $\Sigma$ be an LPV-SSA representation of the form \eqref{equ:alpvss}. It satisfies the regularity certificate if
 \begin{enumerate}
 \item
    $\mathcal{P}$ is convex with non-empty interior;
\item
    in DT, the matrix $A(\bar p)$ is invertible for all $\bar p \in \mathbb{P}$.
\end{enumerate}
\end{definition}

\begin{theorem}[Implication of observability]
\label{min:compare:col1}
 Let $\Sigma$ be an observable LPV-SSA representation of the form \eqref{equ:alpvss} such that $\Sigma$ satisfies the regularity certificate.
 There is at least one scheduling trajectory $p_\mathrm{o} \in \mathcal{P}$ and $t_\mathrm{o} > 0$ such that for any two states $x_1,x_2$ of $\Sigma$, $\mathfrak{Y}_{\Sigma,x_1}=\mathfrak{Y}_{\Sigma,x_2}$ if and only if 
 \[\mathfrak{Y}_{\Sigma,x_1}(0,p_\mathrm{o})(\tau) = \mathfrak{Y}_{\Sigma,x_2}(0,p_\mathrm{o})(\tau), \quad \forall \tau \in [0,t_\mathrm{o}].  \]
 In CT, $p_\mathrm{o}$ can be chosen to be analytic.
\end{theorem}
The proof is given in the Appendix. 
 We will call such a $p_\mathrm{o}$  to be a \emph{revealing scheduling trajectory} on $[0,t_\mathrm{o}]$.
\begin{corollary}[Observability revealing]
\label{min:compare:col11}
If $\Sigma$ is observable and it satisfies the regularity certificate,
then there exists a revealing
 $p_\mathrm{o} \in \mathcal{P}$ and a $t_\mathrm{o} > 0$, such that
 the LTV representation associated with $\Sigma$ and $p_\mathrm{o}$ is completely observable on $[0,t_\mathrm{o}]$.
 \end{corollary}
 By duality, the following holds true:
\begin{corollary}[Reachability revealing]
\label{min:compare:col2}
If $\Sigma$ is span-reachable from $x_\mathrm{o}=0$ and  $\Sigma$ satisfies the regularity certificate, 
 then there exists  exists a revealing
 $p_\mathrm{r} \in \mathcal{P}$ and a $t_\mathrm{r} > 0$, such that
 the LTV representation associated with $\Sigma$ and $p_\mathrm{r}$ is completely controllable on $[0,t_\mathrm{r}]$.
\end{corollary}
Notice that LPV-SSA representations can be viewed as a subclass of LPV state-space representations according to \cite{Toth2010SpringerBook}.
 Theorem \ref{min:compare} presented below allows us to 
relate the minimality concept of Definition \ref{def:min} with the concept of minimality defined in \cite{Toth2010SpringerBook}.  Notice that these two definitions of minimality are not a-priori the same.
Recall from \cite[Definition 3.37, 3.34]{Toth2010SpringerBook} 
the definition of structural reachability and
structural observability. 
Recall from \cite{Toth2010SpringerBook} that minimal state-space realizations are
structurally observable and structurally reachable.
\begin{theorem}[Implication of structural properties]
\label{min:compare}
 If $\Sigma$ satisfies the regularity certificate, then
 \begin{itemize}
\item if it is observable, then it is structurally state-observable. 
\item if is  span-reachable from $x_\mathrm{o}=0$, then it is  structurally state reachable.
 \end{itemize}
\end{theorem}
\begin{corollary}[Joint minimality]
\label{min:compare:col3}
 If $\Sigma$ satisfies the regularity certificate and it 
 is weakly minimal w.r.t.  $x_\mathrm{o}=0$,
 then $\Sigma$ is also jointly state minimal in the sense of \cite{Toth2010SpringerBook}.
\end{corollary}

Finally, we can  supply the necessary and sufficient conditions for the existence of an LPV-SSA realization for a given input-output function. These conditions and the resulting realization algorithm will utilize the previously introduced concept of IIR  and the corresponding Markov parameters. More precisely, this characterization will be achieved by constructing a Hankel matrix from the Markov parameters and by proving that $\mathfrak{F}$ has an LPV-SSA realization if and only if the rank of the aforementioned Hankel-matrix is finite.
Note that 
in general, the existence of an IIR  and the corresponding Markov parameters for a given input-output function $\mathfrak{F}$, are only necessary for the  existence of a finite order LPV-SSA representation.

%
%

In order to define the Hankel-matrix of $\mathfrak{F}$, a lexicographic ordering on the set  $\mathcal{S}(\mathbb{I}_0^{n_\mathrm{p}})$ (all possible sequences of the scheduling dependence) must be introduced.
\begin{definition}[Ordering of sequences]
\label{rem:lex:def}
 Recall that $\mathbb{I}_0^{n_\mathrm{p}}=\{0,\cdots, \QNUM \}$. Then, the lexicographic ordering $\prec$ on  $\mathcal{S}(\mathbb{I}_0^{n_\mathrm{p}})$ can be defined as follows. For any $s,r \in  \mathcal{S}(\mathbb{I}_0^{n_\mathrm{p}})$,  $r \prec s$ holds if either
 \begin{enumerate}
 \item[(i)] $|r| < |s|$ (smaller length), or
\item[(ii)] $0 < |r|=|s|=n$, 
 and the following holds
 \begin{equation}
r = r_1\cdots r_n, \quad s = s_1\cdots s_n,  \quad \ r_i, s_j \in \mathbb{I}_0^{n_\mathrm{p}}
\end{equation}
 and for some $l \in \{1,\cdots, n\}$, $r_l < s_l$ with the usual ordering of integers and $r_i = s_i$ for $i=1, \ldots, l-1$. 
 \end{enumerate}
 \end{definition}
  Note that $\prec$ is a complete ordering on $\mathcal{S}(\mathbb{I}_0^{n_\mathrm{p}})$, i.e., all sequences $s^{(i)}\in\mathcal{S}(\mathbb{I}_0^{n_\mathrm{p}})$ are ordered as $\epsilon= s^{(0)}\prec s^{(1)} \prec s^{(2)}\ \ldots $. Furthermore, for all $s,r \in  \mathcal{S}(\mathbb{I}_0^{n_\mathrm{p}})$, $s\prec sr$ if $r \neq \epsilon$.
Then, the so called Hankel-matrix of $\mathfrak{F}$ both in CT and DT can be defined as follows.

\begin{definition}[Hankel matrix] \label{def:Hank} Consider the input-output function $\mathfrak{F}$ which has an IIR. 
The Hankel-matrix $\mathcal{H}_\mathfrak{F}$ associated with $\mathfrak{F}$ is defined as the infinite matrix
 \begin{equation*}
\mathcal{H}_\mathfrak{F} = 
\begin{bmatrix}
     \theta_{\mathfrak{F}}(s^{(0)}s^{(0)}) & \theta_{\mathfrak{F}}(s^{(1)}s^{(0)}) & \cdots & \theta_{\mathfrak{F}}(s^{(\tau)}s^{(0)}) & \cdots \\
     \theta_{\mathfrak{F}}(s^{(0)}s^{(1)}) & \theta_{\mathfrak{F}}(s^{(1)}s^{(1)}) & \cdots & \theta_{\mathfrak{F}}(s^{(\tau)}s^{(1)}) & \cdots \\
     \theta_{\mathfrak{F}}(s^{(0)}s^{(2)}) & \theta_{\mathfrak{F}}(s^{(1)}s^{(2)}) & \cdots & \theta_{\mathfrak{F}}(s^{(\tau)}s^{(2)}) & \cdots \\
     \vdots  & \vdots & \cdots & \vdots & \cdots 
   \end{bmatrix}
 \end{equation*}
where a $n_\mathrm{y} (\QNUM+1) \times (n_\mathrm{u} (\QNUM+1)+1)$ block of $\mathcal{H}_\mathfrak{F} $ in the block row $i$ and block column $j$ equals the Markov-parameter $\theta(s)$, where $s = s^{(j)}s^{(i)}\in \mathcal{S}(\mathbb{I}_0^{n_\mathrm{p}})$ is the concatenation of the sequences $ s^{(i)}$ and $ s^{(j)}$.
\end{definition}
\begin{theorem}[Existence of realization]
\label{theo:exist} An input-output function $\mathfrak{F}$ of the form \eqref{equ:iofunction} has a LPV-SSA  realization,  if and only if
$\mathfrak{F}$ has an IIR and
\begin{equation}
  \rank{ ( \mathcal{H}_\mathfrak{F} ) }=n_\mathfrak{F} < \infty .
\end{equation}
Any minimal LPV-SSA realization of $\mathfrak{F}$ has a state dimension which equal to  $ n_\mathfrak{F}$.
\end{theorem}
The proof is given in the Appendix. 
Note that this is an important point to clarify two things:
\begin{itemize}
\item Not all input-output functions of the form \eqref{equ:iofunction} will have an IIR and hence an LPV-SSA realization. 
  In that case, state-space realization can be only available with a more general form of coefficient dependence, e.g., rational, dynamic, etc. 
\item The dimension $n_\mathfrak{F}$ of a minimal LPV-SSA realization of $\mathfrak{F}$ can be larger than the dimension of an LPV state-space realization which allows dynamic dependence of the state-matrices on the scheduling signal. 
\end{itemize}

An important application of Theorem \ref{theo:exist} is the proof of correctness of the Ho-Kalman-like realization algorithm for LPV-SSA forms, \emph{e.g.}, in \cite{TAW12} and the validity of the underlying assumptions of LPV subspace schemes \cite{Wingerden09,Verdult02,VV05}. Notice that similar results have been shown for linear switched systems in \cite{Pet11,PetCocv11,Pet12}. 

Let us  complete our results by briefly reviewing the Ho-Kalman-like realization algorithm for LPV-SSA forms. For the sequence set $\mathcal{S}(\mathbb{I}_0^{n_\mathrm{p}})$ and a given $n\in\mathbb{N}$, let $\mathrm{Car}_{n}(\mathcal{S}(\mathbb{I}_0^{n_\mathrm{p}}))$ be the number of all sequences $s\in \mathcal{S}(\mathbb{I}_0^{n_\mathrm{p}})$ with length at most $n$, i.e., $|s|\leq n$. Due to the properties of the lexicographic ordering, it follows that if  $N=\mathrm{Car}_{n}(\mathcal{S}(\mathbb{I}_0^{n_\mathrm{p}}))$, then
 \begin{equation}
   \{s^{(0)}, \ldots, s^{(N)} \} = \{ s \in \mathcal{S}(\mathbb{I}_0^{n_\mathrm{p}}) \mid |s| \le n \} .
 \end{equation}
For a given $n,m\in\mathbb{N}$, now we can denote by $\mathcal{H}_{\mathfrak{F}}(n,m)$ the $N n_\mathrm{y} (\QNUM+1) \times M (n_\mathrm{u} (\QNUM+1)+1)$ upper-left sub-matrix of $\mathcal{H}_{\mathfrak{F}}$ with   $N=\mathrm{Car}_{n}(\mathcal{S}(\mathbb{I}_0^{n_\mathrm{p}}))$ and  $M=\mathrm{Car}_{m}(\mathcal{S}(\mathbb{I}_0^{n_\mathrm{p}}))$. 
Consider a LPV-SSA $\Sigma$ and pick an initial state $x_\mathrm{o} \in \X$ of $\Sigma$. 
Let $\mathcal{O}_n$ be the $n$-step extended observability matrix of $\Sigma$ and let  $\mathcal{R}_m$ be the $m$-step extended reachability matrix of $\Sigma$ w.r.t.\ $x_\mathrm{o}$. 
Then, the Hankel matrix $\mathcal{H}_{\mathfrak{Y}_{\Sigma,x_\mathrm{o}}}$ of $\Sigma$ can be obtained from
$\mathcal{O}_n \mathcal{R}_m$ by rearranging its rows and columns 
This observation can be used to derive a Kalman-Ho-like realization algorithm. This algorithm is presented in Algorithm \ref{alg0}.  
\begin{algorithm}
\caption{Ho-Kalman realization}
\label{alg0}
{\fontsize{9.75}{9.75}\selectfont
\begin{algorithmic}[1]
\REQUIRE size parameters $n,m\in\mathbb{N}$ with $m=n+1$, a Hankel matrix $\mathcal{H}_{\mathfrak{F}}(n,m)$ for an input-output function $\mathfrak{F}$.
\STATE \emph{Singular value decomposition} (SVD) of $\mathcal{H}_{\mathfrak{F}}(n,m)$:
$$\mathcal{H}_{\mathfrak{F}}(n,m)=USV^\top $$
where $S$ is block diagonal with strictly positive elements.
\STATE Let $\hat{\mathcal{O}}=US^{1/2}$ and  $\hat{\mathcal{R}}=S^{1/2}V^\top$ with
$\mathcal{H}_{\mathfrak{F}}(n,m)=\hat{\mathcal{O}} \hat{\mathcal{R}} $.
\STATE Let $\bar{\mathcal{R}}$ be the first $\mathrm{Car}_{n}(\mathcal{S}(\mathbb{I}_0^{n_\mathrm{p}})) n_\mathrm{u} (\QNUM+1)$
 columns of $\hat{ \mathcal{R}}$.
 \STATE Let $\tilde{\mathcal{R}}_i=[\begin{array}{ccc} R^{(s^{0)}i)} & \cdots & R^{(s^{(N)}i)}  \end{array}]$, where  $N=\mathrm{Car}_{n}(\mathcal{S}(\mathbb{I}_0^{n_\mathrm{p}}))$ and 
    \( \hat{\mathcal{R}} = \begin{bmatrix} R^{(s^{(0)})} & \cdots &  R^{(s^{(M)})} \end{bmatrix} \)
  is a partitioning of $\hat{\mathcal{R}}$ such that $M=\mathrm{Car}_{m}(\mathcal{S}(\mathbb{I}_0^{n_\mathrm{p}}))$ and  each 
   $n_\mathrm{x} \times (n_\mathrm{u} (n_\mathrm{p}+1)+1)$ block $R^{(s^{(i)})}$ is associated with $s^{(i)}$ in  $\mathcal{S}(\mathbb{I}_0^{n_\mathrm{p}})$.
 Note that $\tilde{\mathcal{R}}_i$ can be viewed as the matrix composed of some left-shifted blocks of  $\hat{\mathcal{R}}$.
\RETURN: $\Sigma=\{A_i,B_i,C_i,0\}_{i=0}^{\QNUM}$ and $x_{\mathrm o}$ such that 
\begin{itemize}
\item $\begin{bmatrix}  x_\mathrm{o}\! &\! B_0\! &\! \cdots\! &\! B_{\QNUM} \end{bmatrix}$: the first $n_\mathrm{u}( \QNUM+1)+1$ columns of $\hat{\mathcal{R}}$   
\item $\begin{bmatrix} C_0^\top\! &\! C_1^\top\! &\! \cdots\! &\! C_{\QNUM}^\top \end{bmatrix}^\top$: the first $n_\mathrm{y} (\QNUM+1)$ rows of $\hat{\mathcal{O}}$,
\item  $A_i=\tilde{\mathcal{R}}_{i} \bar{\mathcal{R}}^{\dag}$ where $ \bar{\mathcal{R}}^{\dag}$ is the Moore-Penrose pseudo-inverse.
\end{itemize}
\end{algorithmic}
}
\end{algorithm}

In order to explain the properties of the LSS-SSA returned by Algorithm \ref{alg0}, we introduce the notion of a partial realization.
\begin{definition}[Partial realization]
Let $\mathfrak{F}$ be an input-output function admitting an IIR, and let $H_\mathfrak{F}$ be its Hankel matrix as defined in Definition \ref{def:Hank}.
The LPV-SSA $\Sigma$ is an $n$-moment partial realization of $\mathfrak{F}$ from the initial state $x_{\mathrm o}$, if 
$\forall s \in \Words, |s| \le n: \theta_{\mathfrak{F}}(s)=\theta_{\mathfrak{Y}_{\Sigma,x_{\mathrm o}}}(s)$. 
We say that $\Sigma$ is a $n$-moment partial realization of $\mathfrak{F}$, if 
there exists an initial state $x_\mathrm{o}\in \mathbb{X}$ such that $\Sigma$ is an $n$-moment partial realization of $\mathfrak{F}$ from $x_{\mathrm o}$.
 \end{definition}
 That is, an LPV-SSA $\Sigma$ is a $n$-moment partial realization of $\mathfrak{F}$ from $x_{\mathrm 0}$ if $\Sigma$ 
 recreates the first $N=\mathrm{Car}_{n}(\mathcal{S}(\mathbb{I}_0^{n_\mathrm{p}}))$ values of the sub-Markov parameters of $\mathfrak{F}$.
 Here, we order the values according to the lexicographic ordering of the arguments. 
 Recall that in DT, the response $\mathfrak{F}(u,p)(t)$ is a polynomial function of $\{p(s),u(s)\}_{s=0}^{t}$ whose coefficients are the sub-Markov
 parameters. Similarly, in CT, $\mathfrak{F}(u,p)(t)$ is an infinite sum of iterated integrals of $p,u$ on $[0,t]$, such that the sub-Markov parameters are 
 the coefficients of these iterated integrals. 
 Hence, if some of the sub-Markov parameters of $\mathfrak{F}$ and $\mathfrak{Y}_{\Sigma,x_{\mathrm o}}$ coincide, the intuitively,
 the values of $\mathfrak{F}$ and of $\mathfrak{Y}_{\Sigma,x_{\mathrm o}}$ should be close. In fact, if $\Sigma$ is an $n$-moment partial realization of
 $\mathfrak{F}$ from $x_{\mathrm o}$, then in DT, $\mathfrak{F}(u,p)(t)=\mathfrak{Y}_{\Sigma,x_{\mathrm o}}(u,p)(t)$ for all $t=0,\ldots,n-1$, $p \in \mathcal{P}$,
 $u \in \mathcal{U}$. 
 The LPV-SSA returned by Algorithm \ref{alg0} can then be characterized as follows.  
\begin{theorem}
\label{theo:part_real}
Let $\mathfrak{F}$ be an input-output function and assume that $\mathfrak{F}$ admits a IIR.  Let $\Sigma$ and $x_{\mathrm o}$ 
be the LPV-SSA and initial state respectively returned by Algorithm \ref{alg0}. Then the following holds.
\begin{itemize}
\item $\Sigma$ is a $n$-moment partial realization of $\mathfrak{F}$ from $x_{\mathrm o}$.
\item If $\rank~ H_{\mathfrak{F}}(n,n)=\rank~ H_{\mathfrak{F}}(n+1,n)=\rank~ H_{\mathfrak{F}}(n,n+1)$, then 
      $\Sigma$ is a $2n+1$-moment partial realization of $\mathfrak{F}$ from $x_{\mathrm o}$.
\item
If $\rank~ H_{\mathfrak{F}}(n,n) = \rank H_{\mathfrak{F}}$, then
      $\Sigma$ is a minimal realization of $\mathfrak{F}$ from $x_{\mathrm o}$.
\item
The condition $\rank~  H_{\mathfrak{F}}(n,n)  = \rank~  H_{\mathfrak{F}}$ holds if there exists an LPV-SSA realization of $\mathfrak{F}$ of dimension at most $n+1$. 
\end{itemize}
\end{theorem}
That is, Algorithm \ref{alg0} returns a minimal LPV-SSA realization of $\mathfrak{F}$, if $n$ is large enough. Otherwise, it returns a partial realization. 
Note that Algorithm \ref{alg0} may return a $2n+1$ partial realization, even if $\mathcal{F}$ is not a realizable by an LPV-SSA representation.

\section{Conclusions}\label{para:concl} 
  We have presented a fairly complete realization theory for LPV-SSA representations. We have also compared the obtained results with those of \cite{Tot10}. Note that unlike \cite{Tot10}, we did not use the language of the behavioral
  approach, focusing instead on input-output functions. A behavioral theory in the style of \cite{Tot10} remains a topic of further research.  Important directions for future research include application of the obtained
  results to systems identification and model reduction of LPV-SSA representations.

\appendix
\section{Proof of the main results} 
\label{para:proof}

\subsection{Proof of the results on IIR}
 In this section, we will prove Lemma \ref{lem:extension} and Lemma \ref{lem:realiofunction}.
 However, in order to present the proofs of these results for the CT case, 
 we will have to recall from
 \cite{Isi:Nonlin,WangGenSer} some
 technical facts on generating series (Fliess series) and their input-output functions.
 These facts will be used later on in several proofs.
 To begin with, a generating series over $Q$ is a function
 $c: \mathcal{S}(\mathbb{I}_{0}^{n_\mathrm{p}}) \rightarrow \mathbb{R}$ such that there exists 
 $K,R > 0$ which satisfies $\forall s \in \mathcal{S}(\mathbb{I}_{0}^{n_\mathrm{p}}): |c(s)| \le KR^{|s|}$.
 Let us apply Definition \ref{def:volterra} for all $s \in \mathcal{S}(\mathbb{I}_{0}^{n_\mathrm{p}})$ and 
 $p \in  \mathcal{C}_p(\mathbb{R}_{0}^{+},\mathbb{R}^{n_p})$  to define $(w_s \diamond p)(t,\tau)$ in CT.
 Then define the function $F_{c}:\mathcal{C}_p(\mathbb{R}_{0}^{+},\mathbb{R}^{\QNUM}) \rightarrow \mathcal{C}_p(\mathbb{R}_0^{+},\mathbb{R})$ \emph{generated
 by a generating series} $c$ as
 $F_c(p)(t)=\sum_{v \in \mathcal{S}(\mathbb{I}_0^{n_{\mathrm{p}}}) } c(v)(w_v \diamond p)(t,0)$
 In the sequel, by abuse of notation, following the established tradition of \cite{Isi:Nonlin,WangGenSer} we will denote $F_c(p)$ by $F_c[p]$.
 From \cite{Isi:Nonlin} it follows that  $F_c$ is well defined.
 Note that the growth condition $\forall s \in \mathcal{S}(\mathbb{I}_{0}^{n_\mathrm{p}}): |c(s)| \le KR^{|s|}$ is necessary for $F_c[u]$ to be well defined.

 In the sequel, we will extend the definition of generating series to 
 include matrix and vector valued series.  To this end,  we define a \emph{generating
 series} as a function $c:\mathcal{S}(\mathbb{I}_{0}^{n_\mathrm{p}}) \rightarrow \mathbb{R}^{n_r \times n_l}$ for some
 integers $n_l,n_r > 0$, such that there exists $K,R > 0$: 
 $\forall v \in \mathcal{S}(\mathbb{I}_{0}^{n_\mathrm{p}}): ||c(v)||_{F} \le KR^{|v|}$. 
 Here, $||.||_F$ denotes the Frobenius norm for matrices.
 It is clear that using  any other standard matrix norm would yield an equivalent definition.
 If $n_l=1$, then $c$ is just a vector valued generating series. It is easy to see that
 $c$ is a generating series according to the above definition, if and only if
 each entry of $c$ is a generating series in the sense of \cite{Isi:Nonlin}.

 Hence, we can define
 $F_c:\mathcal{C}_p(\mathbb{R}_0^{+},\mathbb{R}^{n_p}) \rightarrow \mathcal{C}_p(\mathbb{R}_0^{+},\mathbb{R}^{n_r \times n_l})$ as 
 $F_c[u](t)=\sum_{v \in  \mathcal{S}(\mathbb{I}_0^{n_{\mathrm{p}}}) } c(v)(w_{v} \diamond p)(t,0)$, where the infinite summation is 
 understood in the usually topology of matrices. Clearly, if $c_{i,j}$ denotes the
 $(i,j)$th component of $c$,  $c_{i,j}$ is a generating series in the classical sense and
 $F_{c_{i,j}}[p](t)$ equals the $(i,j)$th entry of the matrix $F_c[p](t)$, $i=1,\ldots,n_r$,
 $j=1,\ldots,n_l$. 

 Although generating series were originally defined for CT, by a 
 slight abuse of terminology, we will use them for the DT case as well.
 This will allow us to unify the terminology. That is, a function $c: \mathcal{S}(\mathbb{I}_0^{n_{\mathrm{p}}}) \rightarrow \mathbb{R}^{n_r \times n_l}$ will be called a generating series, and  the input-output
 function generated by $c$ will be defined as the function
 $F_{c}:(\mathbb{R}^{n_p})^{\mathbb{N}} \rightarrow \mathcal{Y}=Y^{\mathbb{N}}$ such that
 $F_{c}(p)(t)=\sum_{v \in \mathcal{S}(\mathbb{I}_0^{n_{\mathrm{p}}})} c(v)(w_v \diamond p)(t-1,0) = \sum_{q_1 \cdots q_t \in \mathbb{I}_0^{n_{\mathrm{p}}}} c(q_1\cdots q_t)p_{q_1}(0)\cdots p_{q_t}(t-1)$. 
 Similarly to the CT case, by abuse of notation, following the established tradition of \cite{Isi:Nonlin,WangGenSer} we will denote $F_c(p)$ by $F_c[p]$.
 Notice that for the DT case, we do not have to require the 
 growth condition $||c(v)||_{F} \le KR^{|v|}$, $v  \in \mathcal{S}(\mathbb{I}_{0}^{n_\mathrm{p}})$ to hold, in order for $F_{c}[p]$ to
 be well-defined.

Note that the function $F_c$ is defined on $\mathcal{C}_p(\mathbb{R}_{0}^{+},\mathbb{R}^{\QNUM})$ in CT and
$(\mathbb{R}^{\QNUM})^{\mathbb{N}}$ in DT. Recall that $\mathcal{P}$ denotes $\mathcal{C}_p(\mathbb{R}_0^{+}, \mathbb{P})$ in CT, and it denotes $(\mathbb{P})^{\mathbb{N}}$ in DT.  Hence, in general, $\mathcal{P}$ is a proper subset of the domain of definition $F_c$. However, 
if $\mathbb{P}$ contains an affine basis, then the restriction of $F_c$ to $\mathcal{P}$ determines $c$ uniquely.
 \begin{lemma}
 \label{conv_ser:uniq}
   In CT and DT the following holds.
  Assume that $\mathbb{P} \subseteq \mathbb{R}^{\QNUM}$ contains an affine basis of $\mathbb{R}^{\QNUM}$. Then for any
  two generating series $c_1,c_2$,
  \[ (\forall p \in \mathcal{P}: F_{c_1}[p]=F_{c_2}[p]) \implies c_1=c_2. \]
 \end{lemma}
 Note that for $\mathbb{P}=\mathbb{R}^{n_p}$ and CT, the statement of Lemma \ref{conv_ser:uniq}
 is a well-known, see \cite{SonWangObs,Son79b}.
  \begin{proof}
    For $i=1,2$ and integer $k > 0$ define the map $G_{i,k}$ on $\mathbb{R}^{\QNUM k}$ by
  \[ G_{i,k}(p_1,\ldots,p_k)=\sum_{q_1\cdots q_k \in \mathbb{I}_0^{n_{\mathrm{p}}}} c_i(q_1\cdots q_k)p_{1,q_1}\cdots p_{k,q_k}, \]
  where $p_{l,0}=1$ and $p_{l}=(p_{l,1},\ldots,p_{l,n_p})^T \in \mathbb{R}^{\QNUM}$,   $l=1,\ldots,k$.
  We will show that if $F_{c_1}[p]=F_{c_2}[p]$ for all $p \in \mathcal{P}$, then $G_{1,k}(p_1,\ldots,p_k)=G_{2,k}(p_1,\ldots,p_k)$ for all
  $p_1,\ldots,p_k \in \mathbb{P}$, and for all $k > 0$. 

  For DT, notice that $F_{c_i}[p](k)=G_{i,k}(p(0),\ldots,p(k-1))$ for all $p \in \mathcal{P}$, $k > 0$, hence in this case,
  clearly $\forall p \in \mathcal{P}: F_{c_1}[p]=F_{c_2}[p]$ implies $G_{1,k}(p_1,\ldots,p_k)=G_{2,k}(p_1,\ldots,p_k)$ for all
  $p_1,\ldots,p_k \in \mathbb{P}$, and for all $k > 0$.


  For CT, 
  consider a piecewise-constant $p \in \mathcal{P}$,
 i.e. assume that there exists $0 < t_1, \cdots, t_k \in \mathbb{R}$,
 such that $p(s)=p_i \in \mathbb{P}$, $s \in [\sum_{j=1}^{i-1} t_i, \sum_{j=1}^{i} t_i)$, $i=1,\ldots,k$. 
 From \cite[Lemma 2.1]{SonWangObs} and Lemma \ref{lemma:gen_series} it then follows that
 $F_{c_i}[t_1+\cdots+t_k]$, $i=1,2$  are analytic functions of $t_1,\ldots,t_{k}$, and
 \begin{equation}
\label{lem:extension:eq1}
 \begin{split}
    &  \frac{\partial^{k}}{\partial t_1,\ldots,\partial t_k}
     F_{c_i}[p](t_1+\cdots + t_k)|_{t_1=\cdots=t_k=0}= \\
    & =G_{i,k}(p_1,\ldots,p_k)
  \end{split}
  \end{equation}
  If $\forall p \in \mathcal{P}: F_{c_1}[p]=F_{c_2}[p]$, then  $\frac{\partial^{k}}{\partial t_1,\ldots,\partial t_k} F_{c_1}[p](t_1+\cdots + t_k)|_{t_1=\cdots=t_k=t_{k+1}=0} = \frac{\partial^{k}}{\partial t_1,\ldots,\partial t_{k}} F_{c_2}[p](t_1+\cdots + t_k)|_{t_1=\cdots=t_k=0}$,  for any piecewise-constant $p \in \mathcal{P}$, and hence by \eqref{lem:extension:eq1},
  $G_{1,k}(p_1,\ldots,p_k)=G_{2,k}(p_1,\ldots,p_k)$ for all $p_1,\ldots,p_k \in \mathbb{P}$, 

  To conclude the proof, we show that $G_{1,k}(p_1,\ldots,p_k)=G_{2,k}(p_1,\ldots,p_k)$ for all
  $p_1,\ldots,p_k \in \mathbb{P}$, and for all $k > 0$ implies that $c_1=c_2$. 
  Notice that $c_i(q_1\cdots q_k)=G_{i,k}(e_{q_1}\cdots e_{q_k})$ for all $q_1,\ldots,q_k \in \mathbb{I}_0^{\QNUM}$, where $e_0=0$ and
  $e_i$ is the $i$th standard basis vector of $\mathbb{R}^{\QNUM}$.
  Consider an affine basis 
  $\mathbb{B}=\{b_0,\ldots,b_{\QNUM}\} \subseteq \mathbb{P}$ of $\mathbb{R}^{n_p}$.
  Then $e_i=\sum_{j=0}^{\QNUM} \lambda_{i,j}b_j$ for some $\lambda_{i,j} \in \mathbb{R}$, $j \in \mathbb{I}_0^{\QNUM}$ such
  that $\sum_{j=0}^{\QNUM} \lambda_{i,j}=1$ for all $i \in \mathbb{I}_0^{\QNUM}$. Hence,
  $G_{i,k}(e_{q_1},\ldots,e_{q_k})=\sum_{l_1=0}^{\QNUM} \cdots \sum_{l_k=0}^{\QNUM} \lambda_{q_1,l_1} \cdots \lambda_{q_k,l_k} G_{i,k}(b_{l_1},\ldots,b_{l_k})$ for $i=1,2$ and all $q_1,\ldots,q_k \in \mathbb{I}_0^{\QNUM}$.  Since
  for all $q_1,\ldots,q_k \in \mathbb{I}_0^{\QNUM}$, $G_{1,k}(b_{q_1},\ldots,b_{q_k})=G_{2,k}(b_{q_1},\ldots,b_{q_k})$, as
  $b_{q_1},\ldots,b_{q_k} \in \mathbb{P}$, it then follows that
  $c_1(q_1\cdots q_k)=G_{1,k}(e_{q_1},\ldots, e_{q_k})=G_{2,k}(e_{q_1},\ldots,e_{q_k})=c_2(q_1\cdots q_k)$.
  Since $q_1,\ldots,q_k \in \mathbb{I}_0^{\QNUM}$ and $k > 0$ are arbitrary, the claim of the lemma follows.

  \end{proof}
 
 Let $\mathfrak{F}$ be an input-output function which admits a IIR, and recall from Definition \ref{def:grimarkov} the definition of the
 functions $\eta_{i,\mathfrak{F}}:\mathcal{S}(\mathbb{I}_{0}^{n_\mathrm{p}}) \ni v \mapsto \eta_{i,\mathfrak{F}}(v) \in \mathbb{R}^{n_{\mathrm{p}}}$,
 $\theta_{i,j,\mathfrak{F}}: \mathcal{S}(\mathbb{I}_{0}^{n_\mathrm{p}}) \ni v \mapsto \theta_{i,j,\mathfrak{F}}(v) \in \mathbb{R}^{n_\mathrm{y} \times n_\mathrm{u}}$, $i,j \in \mathbb{I}_0^{n_{\mathrm{p}}}$. 
 These functions can be viewed as generating series
 and hence the corresponding functions $F_{\theta_{i,j,\mathfrak{F}}}$
 $F_{\eta_{i,\mathfrak{F}}}$ are well defined, and their domain contains $\mathcal{P}$.
 \begin{proof}[Proof of Lemma \ref{lemma:IIRcon}]
  From the discussion above it follows that
   $\sum_{ s \in \mathcal{S}(\mathbb{I}_{0}^{n_\mathrm{p}})} \eta_{i,\mathfrak{F}}(s)\cdot  (w_s\diamond p)(t,0)=F_{\eta_{i,\mathfrak{F}}}[p](t)$ and 
   $\sum_{s \in \mathcal{S}(\mathbb{I}_{0}^{n_\mathrm{p}})} \theta_{i,j,\mathfrak{F}}(s) p_{j}(\delta) \cdot (w_{s}\diamond p)(t,\delta)=F_{\theta_{q,r,\mathfrak{F}}}[q ^\tau (p)](t-\tau)$, and that the growth condition 
 \eqref{IIR:growth} implies that these infinite sums are absolutely convergent.
 \end{proof}
 The proof of Lemma \ref{lemma:IIRcon} in fact can be generalized to yield the following.
 \begin{lemma}
 \label{lemma:gen_series}
  If $\mathfrak{F}$ admits a IIR, then  for all $p \in \mathcal{P}$, for all $t, \tau \in \mathbb{T}$, $\tau \le t$,  
   \[ \begin{split} 
    & (g_{\mathfrak{F}} \diamond p)(t) = \sum_{i \in \mathbb{I}_0^{n_\mathrm{p}} } p_i(t)F_{\eta_{i,\mathfrak{F}}}[p](t) \\
    & (h_{\mathfrak{F}}  \circ p)(\tau,t) =  \\
    &   \left\{\begin{array}{ll} 
         \sum_{q,r \in \mathbb{I}_0^{\QNUM}} p_{r}(\tau)p_{q}(t)F_{\theta_{q,r,\mathfrak{F}}}[q ^\tau (p)](t-\tau), & \mbox{ CT } \\
     \sum_{q,r \in \mathbb{I}_0^{\QNUM}} p_{r}(\tau)p_{q}(t)F_{\theta_{q,r,\mathfrak{F}}}[q^{\tau+1}(p)](t-\tau-1) & \mbox{ DT}
     \end{array}\right.
   \end{split}
  \]
  Recall that for any $\tau \in \mathbb{T}$, $(q^{\tau}p)(t)=p(\delta+\tau)$ for all $t \in \mathbb{T}$. 
 \end{lemma}
 \begin{proof}[Proof of Lemma \ref{lem:extension}]
  It is easy to see that if 
$\theta_{\mathfrak{F}}=\theta_{\hat{\mathfrak{F}}}$, then
  $\hat{\mathfrak{F}}=\mathfrak{F}$.
  Therefore, we concentrate on proving that $\hat{\mathfrak{F}}=\mathfrak{F}$ implies 
  $\theta_{\mathfrak{F}}=\theta_{\hat{\mathfrak{F}}}$.

  If $\mathfrak{F}=\hat{\mathfrak{F}}$, then $g_{\mathfrak{F}} \diamond p=\mathfrak{F}(0,p)=\hat{\mathfrak{F}}(0,p)=g_{\hat{\mathfrak{F}}} \diamond p$ for all $p \in \mathcal{P}$
   Using this and \eqref{equ:convol} it then follows that $\mathfrak{F}=\hat{\mathfrak{F}}$ implies that 
  for all $u \in \mathcal{U}$,  $p \in \mathcal{P}$, and  $t \in \mathbb{T}$, 
  $\int_0^t (h_{\mathfrak{F}} \diamond p)(\delta,t)u(\delta)d\delta = \int_0^t (h_{\mathfrak{\hat{F}}} \diamond p)(\delta,t)u(\delta)d\delta$   for CT, and
  $\sum_{\delta=0}^{t-1}(h_{\mathfrak{F}} \diamond p)(\delta,t)u(\delta) = \sum_{\delta=0}^{t-1}(h_{\hat{\mathfrak{F}}} \diamond p)(\delta,t)u(\delta)$ for  DT. 

   For DT, one can choose $u$ such that $u(\delta)=e_j$ for some $\delta \in [0,t-1]$ and $u(s)=0$ for all $s \ne \delta \in [0,t-1]$, 
  $j=1,\ldots,n_\mathrm{u}$. Using this remark for all $\delta=0,1,\ldots, t-1$ successively,
  it follows that 
  $\sum_{s=0}^{t-1} (h_{\mathfrak{F}} \diamond p)(s,t)u(s) =  \sum_{s=0}^{t-1} (h_{\mathfrak{F}} \diamond p)(s,t)u(s)$
  implies that $(h_{\mathfrak{F}} \diamond p)(\delta,t)=(h_{\hat{\mathfrak{F}}} \diamond p)(\delta,t)$ for all $\delta \in [0,t]$.

  For CT,  from \cite[Theorem 9.3,Chapter 11]{LangRealAnalysisBook}
  it follows that $\int_0^t (h_{\mathfrak{F}} \diamond p)(\delta,t)u(\delta)d\delta = \int_0^t (h_{\mathfrak{\hat{F}}} \diamond p)(\delta,t)u(\delta)d\delta$ for all $u \in \mathcal{U}$ implies that
$(h_{\mathfrak{F}} \diamond p)(\delta,t)=(h_{\mathfrak{F}} \diamond p)(\delta,t)$ for almost all  
  $\delta \in [0,t]$ and all $t \in \mathbb{R}_{+}$.  Note that 
  by \cite[Lemma 2.2]{WangGenSer} $F_{\theta_{i,j,\mathfrak{F}}}$, $F_{\theta_{i,j,\hat{\mathfrak{F}}}}$ are continuous functions.
  Hence, if $p$ is continous at $0$ from the right, then
  by Lemma \ref{lemma:gen_series},  $(h_{\mathfrak{F}} \diamond p)(\delta,t)$, $(h_{\mathfrak{F}} \diamond p)(\delta,t)$
  are continous at $\delta=0$ from the right, and therefore $(h_{\mathfrak{F}} \diamond p)(\delta,t)=(h_{\mathfrak{F}} \diamond p)(\delta,t)$ for almost all  $\delta \in [0,t]$ implies $(h_{\mathfrak{F}} \diamond p)(0,t)=(h_{\mathfrak{F}} \diamond p)(0,t)$. 

  That is, if $\mathfrak{F}=\hat{\mathfrak{F}}$, then,
  \begin{equation}
  \label{lem:extension:pf:eq2} 
    \begin{split}
      & \forall p \in \mathcal{P}:
       g_{\mathfrak{F}} \diamond p=g_{\hat{\mathfrak{F}}} \diamond p \\
   & \forall p \in \mathcal{P}_c, \forall t \in \mathbb{T}:  
  (h_{\mathfrak{F}} \diamond p)(0,t)=(h_{\hat{\mathfrak{F}}} \diamond p)(0,t),
    \end{split}
  \end{equation}
  where in DT $\mathcal{P}_c=\mathcal{P}$ and in CT $\mathcal{P}_c$ is the set of all $p \in \mathcal{P}$ which are contiuous at $0$ from the right. 
  For a fixed $p \in \mathcal{P}$, $t \in \mathbb{T}$ define the maps
  $G_{p,t}:\mathbb{R}^{\QNUM} \rightarrow \mathbb{R}^{\NY}$
  $H_{p,t}:\mathbb{R}^{\QNUM}  \times \mathbb{R}^{\QNUM} \rightarrow \mathbb{R}^{\NY \times \NU}$ by
  \[ 
     \begin{split} 
     & G_{p,t}(x)=\sum_{i =0}^{\QNUM} x_i (F_{\eta_{i,\mathfrak{F}}}[p](t) - F_{\eta_{i,\hat{\mathfrak{F}}}}[p](t)) \\
     & H_{p,t}(x,\bar{x})=\sum_{q,r=0}^{\QNUM} x_r\hat{x}_q(F_{\theta_{q,r,\mathfrak{F}}}[p](t)-F_{\theta_{q,r,\hat{\mathfrak{F}}}}[p](t))
     \end{split}
  \]
 for $x=(x_1,\ldots,x_{\QNUM})^T$, $\bar{x}=(\bar{x}_1,\ldots,\bar{x}_{\QNUM})^T$, and $x_0=\bar{x}_0=1$.
  We will show that \eqref{lem:extension:pf:eq2} implies that 
  for any $p \in \mathcal{P}$, $t \in \mathbb{T}$,
  \begin{equation}
  \label{lem:extension:pf:eq1}
   \forall b,\hat{b} \in  \mathbb{P}: G_{p,t}(b)=0,\quad H_{p,t}(b,\hat{b})=0 
  \end{equation}
   Assume that \eqref{lem:extension:pf:eq1} holds for all $b,\hat{b} \in \mathbb{P}$ and for any
  $p \in \mathcal{P}$. Let $v_0,\ldots,v_{\QNUM}$ be elements of $\mathbb{P}$ which form an affine basis of
  $\mathbb{R}^{\QNUM}$. 
  Then for any $x \in \mathbb{R}^{\QNUM}$, $\bar{x} \in \mathbb{R}^{\QNUM}$
  there exist $\lambda_{j},\mu_j \in \mathbb{R}$, $j\in \mathbb{I}_0^{\QNUM}$, such that
  $\sum_{j=0}^{\QNUM} \lambda_{j}=1$, $\sum_{j=0}^{\QNUM} \mu_j=1$ and 
  $x=\sum_{j=0}^{\QNUM} \lambda_{j}v_j$, $\bar{x}=\sum_{j=0}^{\QNUM} \mu_j v_j$. 
  Since $v_0,\ldots,v_{\QNUM}$ belong to $\mathbb{P}$, then by \eqref{lem:extension:pf:eq1},  $G_{p,t}(v_{j_1})=0,H_{p,t}(v_{j_1},v_{j_2})=0$, for all $j_1,j_2 \in \mathbb{I}_0^{\QNUM}$.
  Hence, by a direct calculation it follows that 
  $G_{p,t}(x)=G_{p,t}(\sum_{j=0}^{\QNUM} \lambda_{j} v_j)=\sum_{j=0}^{\QNUM} \lambda_{j} G_{p,t}(v_j)=0$ and
  $H_{p,t}(x,\bar{x})=H_{p,t}(\sum_{j=0}^{\QNUM} \lambda_{j} v_j,\sum_{j=0}^{\QNUM} \mu_{j} v_j)=\sum_{j_1,j_2=0}^{\QNUM} \lambda_{j_1} \mu_{j_2} H_{p,t}(v_{j_1},v_{j_2})=0$. Since $x, \bar{x}$ are arbitrary, it then follows that $H_{p,t}=0$, $G_{p,t}=0$, and the latter
  implies that $F_{\eta_{i,\mathfrak{F}}}[p](t) = F_{\eta_{i,\hat{\mathfrak{F}}}}[p](t))$, 
  $F_{\theta_{i,k,\mathfrak{F}}}[p](t)=F_{\theta_{i,k,\hat{\mathfrak{F}}}}[p](t))$ for all $i,j \in \mathbb{I}_0^{\QNUM}$. 
  Indeed, $G_{p,t}(0)=F_{\eta_{0,\mathfrak{F}}}[p](t) - F_{\eta_{0,\hat{\mathfrak{F}}}}[p](t))=0$, 
  $H_{p,t}(0)=(F_{\theta_{0,0,\mathfrak{F}}}[p](t)-F_{\theta_{0,0,\hat{\mathfrak{F}}}}[p](t))=0$, and 
  $\dfrac{dG_{p,t}(x)}{dx_i}=F_{\eta_{i,\mathfrak{F}}}[p](t) - F_{\eta_{i,\hat{\mathfrak{F}}}}[p](t))=0$, 
  $\dfrac{dH_{p,t}(x)}{dx_j}|_{x=0}=(F_{\theta_{0,i,\mathfrak{F}}}[p](t)-F_{\theta_{0,i,\hat{\mathfrak{F}}}}[p](t))=0$,
  $\dfrac{dH_{p,t}(x)}{dx_i dx_k}=(F_{\theta_{i,k,\mathfrak{F}}}[p](t)-F_{\theta_{i,k,\hat{\mathfrak{F}}}}[p](t))=0$,  for all $i,k=1,\ldots,\QNUM$.
  Since $p \in \mathcal{P}$ and $t \in \mathbb{T}$ are arbitrary, by Lemma \ref{conv_ser:uniq} this implies that
  $\eta_{i,\mathfrak{F}}= \eta_{i,\hat{\mathfrak{F}}}$,
  $\theta_{i,k,\mathfrak{F}}=\theta_{i,k,\hat{\mathfrak{F}}}$ for all $i,k \in \mathbb{I}_{0}^{\QNUM}$, i.e.
  $\theta_{\mathfrak{F}}=\theta_{\hat{\mathfrak{F}}}$

We finish the proof by proving that \eqref{lem:extension:pf:eq2} implies  \eqref{lem:extension:pf:eq1}.
%
In the DT case, consider  any $p \in \mathcal{P}$ and $t \in \mathbb{T}$.
Fix any $b \in \mathbb{P}$  and define $\hat{p} \in \mathcal{P}$ by $\hat{p}(t)=b$ and $p(s)=\hat{p}(s)$ for $s=0,\ldots,t-1$.
Notice that by the definition $F_{c}[p](t)=F_{c}[\hat{p}](t)$ for any convergent series $c$. Hence, 
from Lemma \ref{lemma:gen_series} it then follows that
$(g_{\mathfrak{F}} \diamond \hat{p})(t)=(g_{\hat{\mathfrak{F}}} \diamond \hat{p})(t)$ implies 
$G_{p,t}(b)=(g_{\mathfrak{F}} \diamond \hat{p})(t)-(g_{\hat{\mathfrak{F}}} \diamond \hat{p})(t)=0$ for all $b \in \mathbb{P}$.
In order to show that $\forall b,\hat{b} \in \mathbb{P}: H_{p,t}(b,\hat{b})=0$, for any $b,\hat{b} \in \mathbb{P}$ define $\hat{p}  \in \mathcal{P}$ as
$\hat{p}(0)=\hat{b}$, $\hat{p}(t+1)=b$ and $\hat{p}(s)=p(s)$ for all $s=1,\ldots, t$.
Notice that for any convergent series $c$, $F_{c}[p](t)=F_{c}[q_1(\hat{p})](t)$. 
Hence, from Lemma \ref{lemma:gen_series} and
$(h_{\mathfrak{F}} \diamond \hat{p})(0,t+1)=(h_{\hat{\mathfrak{F}}} \diamond \hat{p})(0,t+1)$ and
$H_{p,t}(b,\hat{b})=(h_{\mathfrak{F}} \diamond \hat{p})(0,t+1)-(h_{\hat{\mathfrak{F}}} \diamond \hat{p})(0,t+1)$ it follows that $\forall b,\hat{b} \in \mathbb{P}: H_{p,t}(b,\hat{b})=0$.

For the CT case, 
for any $p \in \mathcal{P}$ and any $b,\hat{b} \in \mathbb{P}$, 
define $\hat{p}_n \in \mathcal{P}$ such that for all $n  \in \mathbb{N}$, $\hat{p}_s(s)=\hat{b}$, if $s \in [0,\frac{1}{n})$, $\hat{p}_n(s)=p(s)$, if
$s \in [\frac{1}{n},t-\frac{1}{n})$ and $\hat{p}_n(s)=b$ if $s \in [t-\frac{1}{n},+\infty)$.
From Lemma \ref{lemma:gen_series} it follows that $H_{\hat{p}_n,t}(b)= (h_{\mathfrak{F}} \diamond \hat{p}_n)(0,t)-(h_{\hat{\mathfrak{F}}} \diamond \hat{p}_n)(0,t)$  and $G_{\hat{p}_n,t}(b)=(g_{\mathfrak{F}} \diamond \hat{p}_n)(t)-(g_{\hat{\mathfrak{F}}} \diamond \hat{p}_n)(t)$.
Notice  that $\hat{p}_n$ is continuous at zero from the right. 
Hence, $(g_{\mathfrak{F}} \diamond \hat{p}_n)(t)=(g_{\hat{\mathfrak{F}}} \diamond \hat{p}_n)(t)$ and  $(h_{\mathfrak{F}} \diamond \hat{p_n})(0,t)=(h_{\hat{\mathfrak{F}}} \diamond \hat{p}_n)(0,t)$. Hence, $H_{\hat{p}_n,t}(b,\hat{b})=0$ and $G_{\hat{p}_n,t}(b)=0$.
It is also easy to see that $\lim_{n \rightarrow \infty} \int_0^{t} \| \hat{p}_n(s)-p(s)\| ds=0$,
i.e. the restriction $\hat{p}_n|_{[0,t]}$ converges to $p|_{[0,t]}$ in the $L^1([0,t],\mathbb{R}^{\QNUM})$. 
From \cite[Lemma 2.2]{WangGenSer} it follows that $\lim_{n \rightarrow \infty} F_c[\hat{p}_n](t)=F_c[p](t)$ for any convergent series $c$. 
Therefore, $H_{p,t}(b,\hat{b})=\lim_{n \rightarrow \infty} H_{\hat{p}_n,t}(b,\hat{b})$ and 
$G_{p,t}=\lim_{n \rightarrow \infty} G_{\hat{p}_n,t}(b)=0$. 
From this and $H_{\hat{p}_n,t}(b,\hat{b})=0$ and $G_{\hat{p}_n,t}(b)=0$, \eqref{lem:extension:pf:eq1} follows.

  Finally, for every $p \in \mathcal{P}_e$, $t \in \mathbb{T}$, define $(g_{\mathfrak{F}_{e}} \diamond p)(t)=\sum_{i \in \mathbb{I}_0^{\QNUM}} p_i(t)F_{\eta_{i,\mathfrak{F}}}[p](t)$ and
  $(h_{\mathfrak{F}_{e}} \diamond p)(\tau,t)=\sum_{q,r \in \mathbb{I}_0^{n_\mathrm{p}}} p_r(\tau)p_q(t) F_{\theta_{q,r,\mathfrak{F}}}[q^{\tau}(p)](t-\tau)$ for the CT case, and
  $(h_{\mathfrak{F}_{e}} \diamond p)(\tau,t)=\sum_{q,r \in \mathbb{I}_0^{n_\mathrm{p}}} p_r(\tau)p_q(t) F_{\theta_{q,r,\mathfrak{F}}}[q^{\tau+1}(p)](t-\tau)$ for the DT case.
  Note that $F_{\eta_{q,\mathfrak{F}}}[p](t)$ and $F_{\theta_{q,r,\mathfrak{F}}}[q^{\delta} p](t-\tau)$, $\delta=\tau,\tau+1$ are defined for all $p \in \mathcal{P}_{e}$,
  both in DT and CT. 
  Thus, if  for any $(u,p) \in \mathcal{U} \times \mathcal{P}_e$ and $t \in \mathbb{T}$, we define
  \[ \mathfrak{F}_e(u,p)(t) = (g_{\mathfrak{F}_{e}} \diamond p)(t) + \left\{\begin{array}{rl}
                  \int_0^{t} (h_{\mathfrak{F}_{e}} \diamond p)(\delta,t)u(\delta) d\delta & \mbox{ CT} \\
                   \sum_{\delta=0}^{t-1} (h_{\mathfrak{F}_{e}} \diamond p)(\delta,t)u(\delta) & \mbox{ DT} 
                  \end{array}\right. ,
\]
then $\mathfrak{F}_e$ satisfies the conditions of the lemma.
 \end{proof}
 \begin{proof}[Proof of Lemma \ref{lem:realiofunction}]
  We start by analyzing the input-output function $\mathfrak{Y}_{\Sigma,x_{\mathrm{o}}}$ of $\Sigma$. 
 To this end, for any $(u,p) \in \mathcal{U} \times \mathcal{P}$ and for any
  $t \in \mathbb{T}, 0 \le \tau \le t$, define
  \[ 
     \begin{split}
     & (h_{\Sigma} \diamond p)(t,\tau)= 
        \left\{ \begin{array}{rl}
                  C(p(t))\Phi(t,\tau)B(p(\tau)) & \mbox{ in CT} \\
                  C(p(t))\Phi(t-1,\tau+1)B(p(\tau)) & \mbox{ in DT} 
        \end{array}\right.,  \\
       & (g_{\Sigma} \diamond p)(t)=C(p(t))\Phi(t,0)x_{\mathrm{o}}
    \end{split}
  \]
 where $\Phi(t,\tau)$ is the fundamental matrix of $A(p(t))$, i.e. $\xi \Phi(t,\tau)=A(p(t))\Phi(t,\tau)$, $\Phi(\tau,\tau)=I_{n_x}$.
 For DT, we set $\Phi(t,\tau)=0$ for $\tau > t$. 
It is then easy to see that
 \begin{equation}
 \label{lem:realiofunction:eq1} 
  \begin{split} 
   & \mathfrak{Y}_{\Sigma,x_{\mathrm o}}(u,p)(t)= (g_{\Sigma} \diamond p)(t)+ \\
   & + \left\{\begin{array}{rl} 
      \sum_{\delta=0}^{t-1} (h_{\Sigma} \diamond p)(t,\delta)u(\delta), &  \mbox{ in DT } \\
      \int_{0}^{t} (h_{\Sigma} \diamond p)(t,\delta)u(\delta)d\delta &  \mbox{ in CT } 
   \end{array}\right.
  \end{split}
 \end{equation}
 Consider the bilinear system
 \begin{equation}
 \label{lem:realiofunction:eq2} 
 \begin{split}
      & \xi \eta(\delta)= A_0\eta(\delta)+ \sum_{i=1}^{\QNUM} (A_i\eta(\delta))w_i(\delta) \\
      & y(\delta)=C(p(t))\eta(\delta).
 \end{split}
 \end{equation}
  Set the initial state $\eta(0)$ of \eqref{lem:realiofunction:eq2} to be the 
  $i$th column of $B(p(\tau))$.
  Notice that the $i$th column of $(h_{\Sigma}  \diamond p)(t,\tau)$ is the output of
  \eqref{lem:realiofunction:eq2} a time $t-\tau$ for $w=\sigma_{\tau}p$ in CT, and it is the output
  of \eqref{lem:realiofunction:eq2} at time $t-\tau-1$  for $w=\sigma_{\tau+1}(p)$ in DT.  
  Similarly, if we set  $\eta(0)=x_{\mathrm o}$, then
  $(g_{\Sigma} \diamond p)(t)$ is the output of \eqref{lem:realiofunction:eq2} for $w=p$. 
  From \cite{Son79b,Isi95,isi:tac} it then follows that
  \[  
     \begin{split} 
       & (h_{\Sigma} \diamond p)(\tau,t)= \\
       & \left\{ \begin{array}{ll}
           \sum_{s \in \Words} c(s)(w_{s} \diamond p)(\tau,t), & \mbox{CT} \\
           \sum_{s \in \Words} c(s)(w_{s} \diamond p)(\tau+1,t-1), & \mbox{DT} \\
         \end{array}\right. \\
     & (g_{\Sigma} \diamond p)(t)=\left\{\begin{array}{ll}
               \sum_{s \in \Words} c_0(s)(w_s \diamond p)(0,t), & \mbox{CT} \\ 
               \sum_{s \in \Words} c_0(s)(w_s \diamond p)(0,t-1), & \mbox{ DT} \\ 
       \end{array}\right. 
    \end{split}
  \]
  where $c:\Words \rightarrow \mathbb{R}^{n_y \times n_u}$,
  $c_0:\Words \rightarrow \mathbb{R}^{n_y}$ and
  \begin{align*}
    & c(s)=\sum_{r,q \in \AQ} p_r(t)p_q(\tau) C_{r}A_{s}B_{q} \\
    &  c_0(s)=\sum_{q\in \AQ} p_q(t)C_{q}A_s x_{\mathrm o} 
 \end{align*}
 for all $s \in \Words$, 
  Let us define $\theta_{\mathfrak{Y}_{\Sigma,x_{\mathrm o}}}: \Words \rightarrow \mathbb{R}^{(\QNUM+1) n_\mathrm{y} \times (n_\mathrm{u}(\QNUM+1)+1)}$ 
  as in \eqref{lem:realiofunction:eq0}, i.e. for all $i,j \in \AQ$, $s \in \Words$,
  $\theta_{i,j,\mathfrak{Y}_{\Sigma,x_{\mathrm{o}}}}(s)=C_iA_{s}B_j$, $\eta_{j,\mathfrak{Y}_{\Sigma,x_{\mathrm o}}}=C_jA_s x_{\mathrm{o}}$  and
  $\theta_{\mathfrak{Y}_{\Sigma,x_{\mathrm o}}}(s)$ equals
  \[
   \begin{bmatrix}
            \eta_{0,\mathfrak{Y}_{\Sigma,x_{\mathrm o}}}(s) & \theta_{0,0,\mathfrak{Y}_{\Sigma,x_{\mathrm o}}}(s) & \cdots & \theta_{0,\QNUM,\mathfrak{Y}_{\Sigma,x_{\mathrm o}}}(s) \\ 
            \eta_{1,\mathfrak{Y}_{\Sigma,x_{\mathrm o}}}(s) & \theta_{1,0,\mathfrak{Y}_{\Sigma,x_{\mathrm o}}}(s) & \cdots & \theta_{1,\QNUM,\mathfrak{Y}_{\Sigma,x_{\mathrm o}}}(s) \\ 
            \vdots       & \vdots       & \cdots & \vdots  \\ 
            \eta_{\QNUM,\mathfrak{Y}_{\Sigma,x_{\mathrm o}}}(s) & \theta_{\QNUM,0,\mathfrak{Y}_{\Sigma,x_{\mathrm o}}}(s) & \cdots & \theta_{\QNUM,\QNUM,\mathfrak{Y}_{\Sigma,x_{\mathrm o}}}(s) \\ 
           \end{bmatrix}.
  \]
  Then for all $p \in \mathcal{P}$, 
  define the functions $(h_{\mathfrak{Y}_{\Sigma,x_{\mathrm o}}} \diamond p), (g_{\mathfrak{Y}_{\Sigma,x_{\mathrm o}}} \diamond p)$ as follows: 
  for all $t \in \mathbb{T}$, $0 \le \tau \le t$, define
  \begin{align*}
    & (h_{\mathfrak{Y}_{\Sigma,x_{\mathrm o}}} \diamond p)(\tau,t)=(h_{\Sigma} \diamond p)(\tau,t) \\
   & (g_{\mathfrak{Y}_{\Sigma,x_{\mathrm o}}} \diamond p)(t)=  (g_{\Sigma} \diamond p)(t)
  \end{align*}
  It then follows that $(h_{\mathfrak{Y}_{\Sigma,x_{\mathrm o}}} \diamond p)(\tau,t)$, $(g_{\mathfrak{Y}_{\Sigma,x_{\mathrm o}}} \diamond p)(t)$ and $\mathfrak{F}=\mathfrak{Y}_{\Sigma,x_{\mathrm o}}$
  satisfy \eqref{equ:convol}.
  Notice that, if define $\alpha=\max\{ ||C_q||_F \mid q \in \AQ\} \cup \{||x_{\mathrm o}|| ||B_q||_F \mid q \in \AQ\}$ and $K=\alpha^{2}\sqrt{\QNUM(\QNUM+1)}$, 
  $R=\max_{q \in \AQ} ||A_q||_F$, then 
  $||\theta_{\mathfrak{Y}_{\Sigma,x_{\mathrm o}}}(s)||_F \le KR^{|s|}$
  for all $s \in \Words$.
  Hence, $\mathfrak{Y}_{\Sigma,x_{\mathrm o}}$ has a IIR and 
  $\theta_{\mathfrak{Y}_{\Sigma,x_{\mathrm o}}}$ is the function of  
  sub-Markov parameters.
  
  Assume that $\Sigma$ is a realization of $\mathfrak{F}$. Then $\mathfrak{Y}_{\Sigma,x_{\mathrm o}}=\mathfrak{F}$ for some initial state $x_{\mathrm o}$ of $\Sigma$.
 From Lemma 
  \ref{lem:extension}, $\theta_{\mathfrak{Y}_{\Sigma,x_{\mathrm o}}}=\theta_{\mathfrak{F}}$ and hence $\theta_{\mathfrak{F}}$ satisfies \eqref{lem:realiofunction:eq0}.
  Conversely, assume that $\theta_{\mathfrak{F}}$ 
  satisfies \eqref{lem:realiofunction:eq0}. Then
  $\theta_{\mathfrak{F}}=\theta_{\mathfrak{Y}_{\Sigma,x_{\mathrm o}}}$ and thus by 
  Lemma \ref{lem:extension} $\mathfrak{F}_{\Sigma,x_{\mathrm o}}=\mathfrak{F}$, i.e.
  $\Sigma$ is a realization of $\mathfrak{F}$.
 \end{proof}

\subsection{Relationship between LPV-SSAs and linear switched state-space representations}
\label{para:proof:switch}
\label{para:LSLPV}

 In this section we  state the relationship between the LPV-SSAs and \emph{linear switched state-space representations} (\emph{abbreviated by LSS-SS}). 
This relationship will allow us to prove the results on realization theory of LPV-SSAs. To this end,  we introduce the following notation.
\begin{notation}
 Denote $\mathbb{P}_{sw} = \{ e_0, e_1, \cdots, e_{D}\}$, where $e_0$ is the zero vector\footnote{i.e. all entries of $e_0$ are zero} in $\mathbb{R}^{n_{\mathrm p}}$,  and
 let $\mathcal{P}_{sw}$ either $\mathcal{C}_{p}(\mathbb{R}_{+},\mathbb{P}_{sw})$ (cont. time) or
 $\mathbb{P}_{sw}^{\mathbb{N}}$ (discrete. time).
\end{notation}
 An LSS-SS is just an  LPV-SSA 
for which the space of scheduling variables equals $\mathbb{P}_{sw}$. Then, potential input-output functions of LSS-SSs are functions of the form 
 \begin{equation*}
   \mathfrak{F} : \mathcal{U} \times \mathcal{P}_{sw}  \mapsto \mathcal{Y}
 \end{equation*}
such that $\mathfrak{F}$ admits an IIR. LSS-SSs and their input-output functions in the sense of \cite{Pet12,PetCocv11} correspond to LSS-SSs and their input-output functions in the above sense, if each scheduling variable $e_q$ is identified 
with the discrete mode $q \in \AQ$ 
(here $e_0=0$). 
We refer the reader to \cite{PetCocv11,Pet12} for the notion of realization, minimality, observability, span-reachability and isomorphism for LSS-SSs. 
Alternatively, all these notions are special cases of the corresponding concepts for LPV-SSAs,
if LSS-SSs are identified as a subclass of LPV-SSAs. 
The discussion above prompts us to define the following concept.
\begin{definition}\label{def:switchediofunction}
For each function $\mathfrak{F} : \mathcal{U} \times \mathcal{P} \mapsto \mathcal{Y}$ admitting an IIR, the associated switched input-output function
 $\SWS(\mathfrak{F}) : \mathcal{U} \times \mathcal{P}_{sw} \mapsto \mathcal{Y}$ 
 is defined as follows: if $\mathfrak{F}_{\mathrm e}$ is the extension of $\mathfrak{F}$ to 
$\mathcal{U} \times \mathcal{P}_{\mathrm e}$ as described in Lemma~\ref{lem:extension}, then
$\SWS(\mathfrak{F})$  is the restriction of $\mathfrak{F}_{\mathrm e}$ to $\mathcal{U} \times \mathcal{P}_{sw} \subseteq \mathcal{U} \times \mathcal{P}_e$.
  \end{definition}
Lemma \ref{lem:extension} yields the following corollary.
\begin{corollary}
\label{switchediofunction:col1}
 $\theta_{\mathfrak{F}} = \theta_{\mathfrak{F}_{\mathrm e}} = \theta_{\SWS(\mathfrak{F})}$
\end{corollary}
\begin{proof}[Proof of Corollary \ref{switchediofunction:col1}]
 Lemma \ref{lem:extension} already implies that $\theta_{\mathfrak{F}}=\theta_{\mathfrak{F}_{\mathrm e}}$.
 It is then left to show that $\theta_{\mathfrak{F}_{\mathrm e}}=\theta_{\SWS(\mathfrak{F})}$.
 Note that with $\theta_{\SWS(\mathfrak{F})}=\theta_{\mathfrak{F}_{\mathrm e}}$ $\SWS(\mathfrak{F})$ 
 satisfies the conditions of having a IIR. 
 Notice that the  affine span of
 the elements of $\mathbb{P}_{sw}$ yields the whole space $\mathbb{R}^{\QNUM}$.
 Hence from Lemma \ref{lem:extension} it follows that $\theta_{\SWS(\mathfrak{F})}$ is uniquely
 determined by $\SWS(\mathfrak{F})$, i.e. $\theta_{\SWS(\mathfrak{F})}=\theta_{\mathfrak{F}_{\mathrm e}}$ is the only choice of 
 $\theta_{\SWS(\mathfrak{F})}$ with which $\SWS(\mathfrak{F})$ admits an IIR.
\end{proof}
\begin{corollary}
\label{switchediofunction:col11}
The function $\SWS$ is injective, i.e. $\mathfrak{F}_1=\mathfrak{F}_2 \iff \SWS(\mathfrak{F}_1)=\SWS(\mathfrak{F}_2)$.
\end{corollary}
\begin{proof}
 By Lemma \ref{lem:extension} and
 Corollary \ref{switchediofunction:col1},
 $\mathfrak{F}_1 = \mathfrak{F}_2 \iff \theta_{\mathfrak{F}_1}=\theta_{\mathfrak{F}_2} \iff \theta_{\SWS(\mathfrak{F}_1)}=\theta_{\SWS(\mathfrak{F}_2)} \iff \SWS(\mathfrak{F}_1)=\SWS(\mathfrak{F}_2)$.
\end{proof}

The correspondence between LPV-SSAs and LSS-SSs can now be stated.
\begin{definition}[LSS-SS associated with LPV-SSA]
 Let $\Sigma$ be an LPV-SSAs of the form~\eqref{equ:alpvss}-\eqref{equ:affdep}. 
Then, the LSS-SS $\mathfrak{S}(\Sigma)$ associated with $\Sigma$ is the following LSS-SS:
 \begin{equation}
   \mathfrak{S}(\Sigma) = \left(\mathbb{P}_{sw}, \{(A_i, B_i, C_i,0) \}_{q=0}^{\QNUM} \right) .
 \end{equation}
\end{definition}
The following theorem collects the main properties related to the correspondence between LSS-SSs and LPV-SSAs.
\begin{theorem}\label{th:alpvlsrelation}
  Let $\mathfrak{F}$ be an input-output function of the form \eqref{equ:iofunction} and assume that $\mathfrak{F}$ admits an IIR. 
  Let $\Sigma$ be an LPV-SSA of the form \eqref{equ:alpvss}.
  \begin{enumerate}
  \item
  \label{thm:alpvsrelation:part0}
    For every initial state $x \in \mathbb{X}$ of $\Sigma$, $\SWS(\mathfrak{Y}_{\Sigma,x})=\mathfrak{Y}_{\mathfrak{S}(\Sigma),x}$. 
  \item 
  \label{thm:alpvsrelation:part1}
   $\Sigma$ is a realization of the input-output function $\mathfrak{F}$ from the initial state $x_{\mathrm o}$ 
  if and only if $\mathfrak{S}(\Sigma)$ is a realization of $\SWS(\mathfrak{F})$ from the initial state $x_{\mathrm o}$. (see Def.~\ref{def:switchediofunction}).
  \item 
  \label{thm:alpvsrelation:part2}
    $\dim \mathfrak{S}(\Sigma)=\dim \Sigma$.
  \item 
  \label{thm:alpvsrelation:part3}
    Two  LPV-SSAs  $\Sigma_1$ and $\Sigma_2$ are isomorphic 
  if and only if $\mathfrak{S}(\Sigma_1)$ is isomorphic to $\mathfrak{S}(\Sigma_2)$.
  \item 
  \label{thm:alpvsrelation:part4}
   $\Sigma$ is  span-reachable from $x_{\mathrm o}$ if and only if $\mathfrak{S}(\Sigma)$  is span-reachable from $x_{\mathrm o}$. 
   $\Sigma$ is observable if and only if $\mathfrak{S}(\Sigma)$  is observable.
  \end{enumerate} 
\end{theorem}
\begin{proof}
  \textbf{Proof of Part \ref{thm:alpvsrelation:part0}.}
   From Lemma \ref{lem:realiofunction} it follows that
   $\forall i,j \in \AQ, \forall s \in \Words$, $\eta_{i,\SWS(\mathfrak{Y}_{\Sigma,x})}(s) = \eta_{i,\mathfrak{Y}_{\Sigma,x}}(s)=   C_iA_sx = \eta_{i,\mathfrak{Y}_{\mathfrak{S}(\Sigma),x}}$, 
  and $\theta_{i,j,\SWS(\mathfrak{Y}_{\Sigma,x})}(s) = \theta_{i,j,\mathfrak{Y}_{\Sigma,x}}(s)=   C_iA_sB_j = \theta_{i,j,\mathfrak{Y}_{\mathfrak{S}(\Sigma),x}}$. That it,
   $\theta_{\SWS(\mathfrak{Y}_{\Sigma,x})}=\theta_{\mathfrak{Y}_{\mathfrak{S}(\Sigma),x}}$. 
  Since $\SWS(\mathfrak{Y}_{\Sigma,x})$, $\mathfrak{Y}_{\mathfrak{S}(\Sigma),x}$ are both realizable by $\Sigma$, they admit an IIR, by Lemma \ref{lem:extension}
   $\theta_{\SWS(\mathfrak{Y}_{\Sigma,x})}=\theta_{\mathfrak{Y}_{\mathfrak{S}(\Sigma),x}}$ implies
  $\SWS(\mathfrak{Y}_{\Sigma,x})=\mathfrak{Y}_{\mathfrak{S}(\Sigma),x}$. 
    
  \textbf{Proof of Part \ref{thm:alpvsrelation:part1}.}
   Notice that $\Sigma$ is a realization of $\mathfrak{F}$ from the initial state $x_{\mathrm o}$ 
   if and only if $\mathfrak{Y}_{\Sigma,x_{\mathrm o}}=\mathfrak{F}$. By Corollary \ref{switchediofunction:col11},
   $\mathfrak{Y}_{\Sigma,x_{\mathrm o}}=\mathfrak{F}$ is equivalent to $\SWS(\mathfrak{Y}_{\Sigma,x_{\mathrm o}})=\SWS(\mathfrak{F})$. From Part \ref{thm:alpvsrelation:part0} it follows that
   $\SWS(\mathfrak{Y}_{\Sigma,x_{\mathrm o}})=\SWS(\mathfrak{F})$ is equivalent to 
   $\mathfrak{Y}_{\mathfrak{S}(\Sigma),x_{\mathrm o}}=\SWS(\mathfrak{F})$, and the latter is equivalent to $\mathfrak{S}(\Sigma)$ being a realization of $\SWS(\mathfrak{F})$.


  \textbf{Proof of Part \ref{thm:alpvsrelation:part2}.}
     Follows by noticing that the state-spaces of $\Sigma$ and $\mathfrak{S}(\Sigma)$ are the same.

  \textbf{Proof of Part \ref{thm:alpvsrelation:part3}.}
     Follows by noticing that the system matrices of $\Sigma$ and $\mathfrak{S}(\Sigma)$ are the same.

  \textbf{Proof of Part \ref{thm:alpvsrelation:part4}.}
   First we show that $\Sigma$ is span-reachable from $x_{\mathrm o}$
   if and only if $\SWS(\Sigma)$ is span-reachable from $x_{\mathrm o}$.
   To this end, consider the function input-to-state function
   $\mathfrak{X}_{\Sigma,x_{\mathrm o}}:\mathcal{U} \times \mathcal{P} \rightarrow \mathcal{X}$
   of $\Sigma$. 

   It is easy to see that span reachability of $\Sigma$ is equivalent to
   $\forall \nu \in \X: (\nu^T\mathfrak{X}_{\Sigma,x_{\mathrm o}} = 0 \iff \nu = 0)$. 
   For every $\nu \in \X$, consider the function $\mathfrak{F}_{\nu}(u,p)=\nu^T\mathrm{X}_{\Sigma,x_{\mathrm o}}(u,p)$.
   It is clear that the LPV-SSA $\Sigma_{\nu}$, 
   $\Sigma_{\nu}=(\mathbb{P},\{A_i,B_i,\nu\}_{i=0}^{\QNUM})$,
   is a realization of $\mathfrak{F}_{\nu}$ from the initial state $x_{\mathrm o}$.
   It is easy to see that
   $\mathfrak{F}_{\nu}=0$ if and only if $\theta_{\mathfrak{F}_{\nu}}=\theta_{\SWS(\mathfrak{F}_{\nu})}=0$ and hence
   $\SWS(\mathfrak{F}_{\nu})=0 \iff \mathfrak{F}_{\nu}=0$. But from Part \ref{thm:alpvsrelation:part0}  it follows that
   $\SWS(\mathfrak{F}_{\nu})$ equals the function $\nu^T\mathfrak{X}_{\SWS(\Sigma),x_{\mathrm o}}$.
   Hence,
   $\forall \nu \in \X: (\nu^T\mathfrak{X}_{\Sigma,x_{\mathrm o}} = 0 \iff \nu = 0)$ is equivalent to
   $\forall \nu \in \X: (\nu^T \mathfrak{X}_{\SWS(\Sigma),x_{\mathrm o}} = 0 \iff \nu = 0)$.
   The latter is equivalent to span-reachability of $\SWS(\Sigma)$ from $x_{\mathrm o}$.

   Next, we show that $\Sigma$ is observable if and only if
   $\SWS(\Sigma)$ is observable. To this end, notice
   that 
  for any state $x \in \X$
   $\SWS(\mathfrak{Y}_{\Sigma,x})=\mathfrak{Y}_{\SWS(\Sigma),x}$.
   Hence, from Corollary \ref{switchediofunction:col1} it follows that
   $\mathfrak{Y}_{\Sigma,x_1}=\mathfrak{Y}_{\Sigma,x_2}$ if and only if
   $\mathfrak{Y}_{\SWS(\Sigma),x_1}=\mathfrak{Y}_{\SWS(\Sigma),x_2}$.
   From this it follows that observability of $\Sigma$ and $\SWS(\Sigma)$ are
   equivalent.
\end{proof}
 We have just presented a transformation from LPV-SSAs to LSS-SSs. Next, we present the reverse transformation, mapping LSS-SSs to LPV-SSAs. 
 To this end, let $\mathbb{P} \subseteq \mathbb{R}^{\QNUM}$ be a space of
 scheduling parameters such that the affine span of elements of $\mathbb{P}$ equals $\mathbb{R}^{\QNUM}$.
 \begin{definition}[LPV-SSA associated with LSS-SS]
 Let $\mathcal{H}=(\mathbb{P}_{sw}, (A_i,B_i,C_i,0)_{i=0}^{\QNUM})$ be a LSS-SS.
 Define the LPV-SSA associated with $\mathcal{H}$ as 
 $\LPV(\mathcal{H})=(\mathbb{P},(A_i,B_i,C_i,0)_{i=0}^{\QNUM})$.
 \end{definition}
 It is easy to see that $\SWS(\LPV(\mathcal{H}))=\mathcal{H}$, from which, using
 Theorem \ref{th:alpvlsrelation}, we can deduce the following.
 \begin{corollary}
 \label{th:alpvlsrelation:col1}
  Let $\mathfrak{F}$ be an input-output function of the form \eqref{equ:iofunction} admitting an IIR and let $\mathcal{H}$ be an LSS-SS.
  Then the following hold.
  \begin{itemize}
  \item
 $\mathcal{H}$ is a realization of $\SWS(\mathfrak{F})$ from the initial state $x_{\mathrm o}$,
 if and only if $\LPV(\mathcal{H})$ is a realization of $\mathfrak{F}$ from the initial state $x_{\mathrm o}$.  
 \item
 $\mathcal{H}$ is span-reachable from $x_{\mathrm o}$ (observable), if and only if
 $\LPV(\mathcal{H})$ is span-reachable from $x_{\mathrm o}$ (respectively observable).
 \end{itemize}
 \end{corollary}
 We can derive the following corollary of Theorem \ref{th:alpvlsrelation:col1} and Corollary \ref{th:alpvlsrelation:col1}:
 \begin{corollary}[Minimality of LPV-SSA and LSS-SS]
 \label{th:alpvlsrelation:min}
  An LPV-SSA $\Sigma$ is minimal w.r.t. an initial state $x_{\mathrm o}$ if and only if the LSS-SS $\mathfrak{S}(\Sigma)$ is minimal w.r.t $x_{\mathrm o}$.
 \end{corollary}
 \begin{proof}[Proof of Corollary \ref{th:alpvlsrelation:min}]
   Assume that $\Sigma$ is minimal w.r.t. an initial state $x_{\mathrm o}$, i.e. it is a minimal realization of $\mathfrak{F}=\mathfrak{Y}_{\Sigma,x_{\mathrm o}}$. 
   From Theorem \ref{th:alpvlsrelation} it then follows that $\mathfrak{S}(\Sigma)$ is a realization of $\SWS(\mathfrak{F})$ from the
   initial state $x_{\mathrm o}$.  Assume that $\mathcal{H}^{'}$ is an LSS-SSA and $\mathcal{H}^{'}$ is a realization of $\SWS(\mathfrak{F})$. It then
   follows that $\Sigma^{'}=\LPV(\mathcal{H}^{'})$ is a realization of $\mathfrak{F}$. Since $\Sigma$ is a minimal realization of $\mathfrak{F}$,
   it then follows that $\dim \mathfrak{S}(\Sigma) = \dim \Sigma \le \dim \Sigma^{'} = \dim \mathcal{H}^{'}$. 
   Conversely, assume that  $\mathfrak{S}(\Sigma)$ is minimal w.r.t $x_{\mathrm o}$, i.e. assume that $\mathfrak{S}(\Sigma)$ is a minimal
   realization of $\mathfrak{Y}_{\mathfrak{S}(\Sigma),x_{\mathrm o}}$.  Assume that $\Sigma^{'}$ is an LPV-SSA such that $\Sigma^{'}$ is a realization of
   $\mathcal{F}=\mathfrak{Y}_{\Sigma,x_{\mathrm o}}$. From Theorem \ref{th:alpvlsrelation} it follows that $\mathfrak{S}(\Sigma^{'})$ is a realization of
   $\SWS(\mathcal{F})$. Note that $\mathfrak{Y}_{\mathfrak{S}(\Sigma),x_{\mathrm o}}=\SWS(\mathfrak{F})$, and hence, by minimality of $\mathfrak{S}(\Sigma)$
   w.r.t. $x_{\mathrm o}$, $\dim \Sigma=\dim \mathfrak{S}(\Sigma) \le \dim \mathfrak{S}(\Sigma^{'}) = \dim \Sigma^{'}$. That is,
   $\Sigma$ is indeed a minimal realization of $\mathfrak{F}=\mathfrak{Y}_{\Sigma,x_{\mathrm o}}$. 
 \end{proof}

\subsection{Proofs of the results on Kalman-style realization theory for LPV-SSAs}
\label{para:proof:kalman}
Based on the relationship between LSS-SSs and LPV-SSAs explained 
in the previous section, we can use realization theory of LSS-SSs
\cite{Pet11,PetCocv11,Pet12}
to prove the results of Section \ref{para:kalman}.  
\begin{proof}[Proof of Theorem \ref{theo:min}]
 From Corollary \ref{th:alpvlsrelation:min} it follows that 
 $\Sigma$ is is minimal w.r.t $x_{\mathrm o}$, if and only if
 $\mathfrak{S}(\Sigma)$ is minimal w.r.t $x_{\mathrm o}$. 
 In turn, by \cite[Theorem 3]{Pet12} (DT) or \cite[Theorem 3]{PetCocv11} (CT),
 $\mathfrak{S}(\Sigma)$ is a minimal w.r.t. $x_{\mathrm o}$,  if and only if
 $\mathfrak{S}(\Sigma)$ is observable and span-reachable from $x_{\mathrm o}$.
 From Theorem \ref{th:alpvlsrelation} it follows that
 $\mathfrak{S}(\Sigma)$ is observable and span-reachable from $x_{\mathrm o}$ if and only if
 $\Sigma$ is span-reachable from $x_{\mathrm o}$ and observable.
 Hence, 
 $\Sigma$ is is minimal w.r.t $x_{\mathrm o}$ if and only if 
 $\Sigma$ is span-reachable from $x_{\mathrm o}$ and observable.

 Assume that $\Sigma$ and $\Sigma^{'}$ are minimal w.r.t. $x_{\mathrm o}$ and $x_{\mathrm o}^{'}$, and 
 $\Sigma$ and $\Sigma^{'}$ are weakly equivalent w.r.t. $x_{\mathrm o}$ and $x_{\mathrm o}^{'}$, i.e. $\mathfrak{Y}_{\Sigma,x_{\mathrm o}}=\mathfrak{Y}_{\Sigma^{'},x^{'}_{\mathrm o}}$.  It then follows that $\mathfrak{S}(\Sigma)$ is minimal w.r.t.  $x_{\mathrm o}$ and 
 $\mathfrak{S}(\Sigma^{'})$ is minimal w.r.t. $x_{\mathrm o}^{'}$. Moreover, $\mathfrak{Y}_{\mathfrak{S}(\Sigma),x_{\mathrm o}}=\mathfrak{Y}_{\mathfrak{S}(\Sigma^{'}),x^{'}_{\mathrm o}}$.
 From \cite[Theorem 3]{Pet12}, \cite[Theorem 3]{PetCocv11} it then
 follows $\mathfrak{S}(\Sigma)$ to $\mathfrak{S}(\Sigma^{'})$ are isomorphic, and 
 hence by Theorem \ref{th:alpvlsrelation}, $\Sigma$ and $\Sigma^{'}$ are isomorphic. 
\end{proof}
 Note that if   $T$ is an isomorphism from $\Sigma$ to $\Sigma^{'}$, and $\Sigma$ and $\Sigma^{'}$ are minimal and weakly equivalent w.r.t.
 $x_{\mathrm o}$ and $x^{'}_{\mathrm o}$ respectively,  then   $Tx_{\mathrm o}=x^{'}_{\mathrm o}$. 
 Indeed, it is not difficult to see that $\mathfrak{Y}_{\Sigma,x_{\mathrm o}}=\mathfrak{Y}_{\Sigma^{'},Tx_{\mathrm{o}}}$.
 Since $\mathfrak{Y}_{\Sigma,x_{\mathrm o}}=\mathfrak{Y}_{\Sigma^{'},x^{'}_{\mathrm o}}$, it then follows that
 $\mathfrak{Y}_{\Sigma^{'},Tx_{\mathrm{o}}}=\mathfrak{Y}_{\Sigma^{'},x^{'}_{\mathrm o}}$, and by observability of $\Sigma^{'}$ this implies that
 $Tx_{\mathrm o}=x^{'}_{\mathrm o}$. 

\begin{proof}[Proof of Theorem \ref{theo:min:strong}]
  It is enough to prove that $\Sigma$ is strongly minimal if and only if it is minimal w.r.t. $0$. The rest then follows from
  Theorem \ref{theo:min:strong}. By definition, if $\Sigma$ is strongly minimal, then it is minimal w.r.t. any initial state, including $0$. 
  Conversely, assume that $\Sigma$ is minimal w.r.t. $0$. Let $x_{\mathrm o}$ be any state of $\Sigma$ and let
  $\mathfrak{F}=\mathfrak{Y}_{\Sigma,x_{\mathrm o}}$. We will show that $\Sigma$ is a minimal realization of $\mathfrak{F}$. To this end, consider
  any LPV-SSA $\Sigma^{'}$ such that $\Sigma^{'}$ is a realization of $\mathfrak{F}$ from some initial state $x_{\mathrm o}^{'}$.  We will show that
  $\dim \Sigma \le \dim \Sigma^{'}$. 
  To this end, notice that 
  $\mathfrak{Y}_{\Sigma^{'}x^{'}_{\mathrm o}}=\mathfrak{F}=\mathfrak{Y}_{\Sigma,x_{\mathrm o}}$, and that for all $u \in \mathcal{U}$, $p \in \mathcal{P}$,
  \( \mathfrak{Y}_{\Sigma^{'},0}(u,p)=\mathfrak{Y}_{\Sigma^{'},x_{\mathrm o}^{'}}(u,p)-\mathfrak{Y}_{\Sigma^{'},x^{'}_{\mathrm o}}(0,p)$ and
   \( \mathfrak{Y}_{\Sigma,x_{\mathrm o}}(u,p)-\mathfrak{Y}_{\Sigma,x_{\mathrm o}}(0,p)= \mathfrak{Y}_{\Sigma,0}(u,p) \).
  From this remark it follows that if $\mathfrak{Y}_{\Sigma^{'}x^{'}_{\mathrm o}}=\mathfrak{Y}_{\Sigma,x_{\mathrm o}}$, then 
  $\mathfrak{Y}_{\Sigma,0}=\mathfrak{Y}_{\Sigma^{'},0}$. Since $\Sigma$ is minimal w.r.t. $0$, it then implies that $\dim \Sigma \le \dim \Sigma^{'}$. 
\end{proof}
\begin{proof}[Proof of Theorem \ref{theo:reachobs:rank:cond}]
For DT, from \cite[Theorem 4]{Pet12} it follows
that $\Rank \mathcal{R}_{\NX-1}=\NX$ if equivalent to $\mathfrak{S}(\Sigma)$ being 
span-reachable from $x_0$, and $\Rank \mathcal{O}_{\NX-1}=\NX$ is equivalent to
observability of $\SWS(\Sigma)$, as in the terminology of
\cite[Theorem 4]{Pet12}, 
$\mathrm{IM} \mathcal{R}_{\NX-1}$ is the image of the span-reachability matrix of $(\mathfrak{S}(\Sigma),x_0)$, and
$\ker \mathcal{O}_{\NX-1}$ is the  kernel observability matrix of $(\mathfrak{S}(\Sigma),x_0)$.
Note that in \cite{Pet12} the definition of a linear switched system included the
initial state.

For CT,
from \cite[Proposition 33]{Pet06}, when applied to
 the rational representation associated with $(\mathfrak{S}(\Sigma),\mu)$, $\mu:\{f\} \ni f \mapsto x_{\mathrm o}$,   it follows that 
$\mathrm{Im} \mathcal{R}_{\NX-1}$ equals $WR(x_{\mathrm o})=\SPAN\{\hat{A}_{q_1}\cdots \hat{A}_{q_k}x \mid q_1,\ldots, q_k \in Q, k \ge 0, \mbox{ and } x=x_0 \mbox{ or } \exists q \in Q: x \in \mathrm{Im} \hat{B}_q\}$, 
where $\hat{A}_0=A_0, \hat{B}_0=B_0$, and $\hat{A}_i=A_i-A_0$, $\hat{B}_i=B_i-B_0$ for all $i \in \AQ$, $i > 0$.
Here, if $k=0$, then $\hat{A}_{q_1}\cdots \hat{A}_{q_k}$ is interpreted as the identity matrix.
But from \cite[Proposition 1]{PetCocv11} it follows that
$\mathfrak{S}(\Sigma)$ is span-reachable from $x_0$ if and only if $WR(x_{\mathrm o})=\X$. Hence,
by Theorem \ref{th:alpvlsrelation}, $\Sigma$ is span-reachable from $x_{\mathrm o}$ if and only if
$\mathrm{Im} \mathcal{R}_{\NX-1}=WR(x_{\mathrm o})=\X$. The latter condition is equivalent to
$\Rank \mathcal{R}_{\NX-1}=\NX$.
Similarly, from \cite[Proposition 34]{Pet06}, applied to 
 the rational representation associated with $(\mathfrak{S}(\Sigma),\mu)$, $\mu:\{f\} \ni f \mapsto 0$,   it follows that 
$\ker \mathcal{O}_{\NX-1}$ equals $O=\bigcap_{q \in Q} \ker C_q \cap \bigcap_{k=1}^{\infty} \bigcap_{q_1\cdots q_k \in Q} \ker \hat{C}_q\hat{A}_{q_k}\cdots \hat{A}_{q_1}$,
where $\hat{A}_0=A_0, \hat{C}_0=C_0$, and $\hat{A}_i=A_i-A_0$, $\hat{C}_i=C_i-C_0$ for all $i \in \AQ$, $i > 0$.

From \cite[Theorem 2]{PetCocv11} it follows that $\mathfrak{S}(\Sigma)$ is observable if and only if
$O=\{0\}$.  
Hence, by Theorem \ref{th:alpvlsrelation}, $\Sigma$ is observable if and only if
$\ker \mathcal{O}_{\NX-1}=O=\{0\}$. The latter condition is equivalent to
$\Rank \mathcal{O}_{\NX-1}=\NX$.
\end{proof}
\begin{proof}[Proof of Corollary \ref{lem:min}]
 Define the spaces $V_1=\mathrm{Span}\{b_1,\ldots,b_{r_m}\}$, $V_2=\mathrm{Span}\{b_{r_m+1},\ldots,b_r\}$ and $V_3=\mathrm{Span}\{b_{r+1},\ldots,b_{n_{\mathrm x}}\}$. It then follows
 that $V_1+V_2$ is $A_i$-invariant, and $V_2$ is $A_i$ invariant for all $i \in \AQ$, since $\mathrm{Im} \{\mathcal{R}_{n_{\mathrm x}-1}\}$ and $\ker \{\mathcal{O}_{n_{\mathrm x}-1}\}$ 
 are 
 $A_i$ invariant subspaces and $V_1+V_2=\mathrm{Im} \{\mathcal{R}_{n_{\mathrm x}-1}\}$, and $V_2=\mathrm{Im} \{\mathcal{R}_{n_{\mathrm x}-1}\} \cap \ker  \{\mathcal{O}_{n_{\mathrm x}-1}\}$.
 Moreover, $\mathrm{Im} B_i \subseteq V_1+V_2=\mathrm{Im} \{\mathcal{R}_{n_{\mathrm x}-1}\}$, $i \in \AQ$,  $x_0 \in V_1+V_2=\mathrm{Im} \{\mathcal{R}_{n_{\mathrm x}-1}\}$, and 
 $V_2 \subseteq \ker \{ \mathcal{O}_{n_{\mathrm x}-1}\} \subseteq \ker C_i$, $i \in \AQ$. From this, it follows that $\hat{A}_i,\hat{B}_i,\hat{C}_i$, $i \in \AQ$ and $\hat{x}_0$
 satisfy \eqref{kalman_decomp}. Let $\hat{\mathcal{R}}_k$  be the $k$th step extended reachability matrix of $\hat{\Sigma}$ for $x_{\mathrm o}$, and let $\hat{\mathcal{O}}_k$
 be the $k$th step extended observability matrix of $\hat{\Sigma}$. Similarly,
 let  $\mathcal{R}_k^{\mathrm m}$  be the $k$th step extended reachability matrix of $\Sigma^{\mathrm m}$ for $x_{\mathrm o}^{\mathrm m}$ and 
 let $\mathcal{O}_k^{\mathrm m}$ be the $k$th step extended  observability matrices of $\Sigma^{\mathrm m}$. By induction on $k$, it 
 follows that 
 \[ 
   \begin{split}
    & \hat{\mathcal{R}}_k=\begin{bmatrix} \mathcal{R}_k^{\mathrm m} \\ R_k \\ 0 \end{bmatrix}, ~  \hat{\mathcal{O}}_k=\begin{bmatrix} \mathcal{O}_k^{\mathrm m} & 0 & O_k \end{bmatrix}, 
   \end{split}
 \]
 where $R_k$ is a suitable matrix with $r-r_m$ rows, and $O_k$ is a suitable matrix with $n-r$ columns. Since $\rank\{ \hat{\mathcal{R}}_{n_{\mathrm x}-1}\}=r$ and 
 $\mathcal{R}_{n_{\mathrm x}-1}^{\mathrm m}$ has $r_m$ rows, it then follows that $\rank \{ \mathcal{R}_{n_{\mathrm x}-1}^{\mathrm m}\}=r_m$ from which 
 by Theorem \ref{theo:reachobs:rank:cond} and $\rank~ \{\mathcal{R}_{n_{\mathrm x}-1}\}=\rank \{\mathcal{R}_{r_m-1}\}$ it follows that $\Sigma^{\mathrm m}$ is span-reachable from $x^{\mathrm m}_{\mathrm o}$
 Similarly, since $\rank \{\hat{\mathcal{O}}_{n_{\mathrm x}-1}\}=r$ and 
 $\mathcal{O}_{n_{\mathrm x}-1}^{\mathrm m}$ has $r_m$ colums, it then follows that $\rank \{\mathcal{O}_{n_{\mathrm x}-1}^{\mathrm m}\}=r_m$ from which 
 by Theorem \ref{theo:reachobs:rank:cond} and $\rank \{\mathcal{O}_{n_{\mathrm x}-1}\}=\rank \{ \mathcal{O}_{r_m-1}\}$ it follows that $\Sigma^{\mathrm m}$ is observable.
 Finally, for any $u \in \mathcal{U}$, $p \in \mathcal{P}$, let $\hat{x}$ and $y$ be such that $(\hat{x},y,u,p)$ is a solution of $\hat{\Sigma}$ and $\hat{x}(0)=\hat{x}_{\mathrm o}$.
 Let $z_1$ be the function formed by the first $r_m$ cooordinates of $\hat{x}$, let $z_2$ be formed by the coordinates of $\hat{x}$ ranging from $r_m+1$ to $r$ and let $z_3$
 be formed by the last $n_{\mathrm x}-r$ coordinates of $\hat{x}$. It then follows that
 $\xi z_3(t)=(A^{\prime\prime\prime\prime}_0+\sum_{i=1}^{n_{\mathrm p}}  A^{\prime\prime\prime\prime}_i p_i(t))z_3(t)$, and as $z_3(0)=0$, it then follows that $z_3(t)=0$. 
 Therefore $\xi z_1(t)=(A^{\mathrm m}_0+\sum_{i=1}^{n_{\mathrm p}}  A^{\mathrm m}_i p_i(t))z_1(t) + (B^{\mathrm m}_0+\sum_{i=1}^{n_{\mathrm p}}  B^{\mathrm m}_i p_i(t))u(t)$.
 Moreover, notice that $y(t)=C^{\mathrm m}_iz_1(t)$. Hence, $(z_3,y,u,p)$ is a solution of $\Sigma^{\mathrm m}$ with $z_3(0)=x_{\mathrm o}^{\mathrm m}$. That is,
 $\mathfrak{Y}_{\hat{\Sigma},\hat{x}_{\mathrm o}}=\mathfrak{Y}_{\Sigma^{\mathrm m},x_{\mathrm o}^{\mathrm m}}$. But $\hat{\Sigma}$ and $\Sigma$ are isomorphic, with $T$
 being the isomorphism, and $Tx_{\mathrm o}=\hat{x}_{\mathrm o}$. 
 Hence, $\mathfrak{Y}_{\hat{\Sigma},\hat{x}_{\mathrm o}}=\mathfrak{Y}_{\Sigma, x_{\mathrm o}}=\mathfrak{Y}_{\Sigma^{\mathrm m},x_{\mathrm o}^{\mathrm m}}$.  
 That is, if $\Sigma$ is a realization of $\mathfrak{F}$ from $x_{\mathrm o}$, then $\Sigma^{\mathrm m}$ is a realization of $\mathfrak{F}$. Since according to the discussion above,
 $\Sigma^{\mathrm m}$ is observable and span-reachable from $x_{\mathrm o}^{\mathrm m}$, by Theorem \ref{theo:min} it is a minimal realization of $\mathfrak{F}$. 
\end{proof}
\begin{proof}[Proof of Theorem \ref{min:compare:col1}]
 Notice that $\mathfrak{Y}_{\Sigma,x_1}=\mathfrak{Y}_{\Sigma,x_2}$ is equivalent to
 $\forall p \in \mathcal{P}: \mathfrak{Y}_{\Sigma,x_1}(0,p)=\mathfrak{Y}_{\Sigma,x_2}(0,p)$,
 since $\mathfrak{Y}_{\Sigma,x_i}(u,p)=\mathfrak{Y}_{\Sigma,x_i}(0,p)+\mathfrak{Y}_{\Sigma,0}(u,p)$ for all
 $i=1,2$, $u \in \mathcal{U}$, $p \in \mathcal{P}$.
 Hence, it is enough to show that there exists $t_{\mathrm o} > 0$, $p_o \in \mathcal{P}$,
 such that for all $x_1,x_2 \in \X$, $(\forall p \in \mathcal{P}: \mathfrak{Y}_{\Sigma,x_1}(0,p)=\mathfrak{Y}_{\Sigma,x_2}(0,p)) \iff (\forall \tau \in [0,t_{\mathrm o}]: \mathfrak{Y}_{\Sigma,x_1}(0,p_o)(\tau)=\mathfrak{Y}_{\Sigma,x_2}(0,p_o)(\tau)$, and that for the CT case, $p_o$
 can be chosen to be analytic.
 
 Define for any initial state $x_{\mathrm o}$ of $\Sigma$ the function $s_{\Sigma,x_{\mathrm o}}: \mathcal{P} \rightarrow \mathcal{Y}$ by
  $s_{\Sigma,x_{\mathrm o}}(p)(t)=\int_0^t \mathfrak{Y}_{\Sigma,x_{\mathrm o}}(0,p)(s)ds$, 
 in CT, and $s_{\Sigma,x_{\mathrm o}}(p)(0)=0$ and $s_{\Sigma,x_{\mathrm o}}(p)(t+1)=\mathfrak{Y}_{\Sigma,x_{\mathrm o}}(0,p)(t)$
 in DT, for all $t \in \mathbb{T}$.
 It then follows that $s=s_{\Sigma,x_0}(p)$ 
 is the output of the bilinear system
 \begin{equation}
 \label{eq:bilin:aux}
    \begin{split}
    & \delta x(t) = \sum_{q=1}^{\QNUM} A_qx(t)p_q(t), \quad  \delta z(t)=\sum_{q=1}^{\QNUM} C_qx(t)p_q(t) \\
    & s(t)=z(t) 
   \end{split}
 \end{equation}
  from the initial state $(x^T(0),z(0))^T=(x^T_{\mathrm o},0^T)^T$.

  For any $p \in \mathcal{P}$,
 denote by $s((x_{\mathrm o},z_{\mathrm o}),p)$ the output trajectory of \eqref{eq:bilin:aux} 
  generated from the initial state $(x^T_{\mathrm o},z^T_{\mathrm o})^T$, $x_{\mathrm o} \in \X$, $z_{\mathrm o} \in \mathbb{R}^{\NY}$. 
  We will call \eqref{eq:bilin:aux} observable, if 
  for each pair of distinct states $(x_1,z_1) \ne (x_2,z_2)$, there exists
  $p \in \mathcal{P}$ such that $s((x_1,z_1),p) \ne s((x_2,z_2),p)$.
  Notice that \eqref{eq:bilin:aux} is observable if and only if $\Sigma$ is observable.
  Indeed, 
  $\delta s((x_{\mathrm o},z_{\mathrm o}),p)=\mathfrak{Y}_{\Sigma,x_{\mathrm o}}(p,0)$ and $s((x_{\mathrm o},z_{\mathrm o}),p)(0)=z_0$.
  Hence, if $\Sigma$ is observable and there exists $(x_1,z_1) \ne (x_2,z_2)$ 
  such that $s((x_1,z_1),p) = s((x_2,z_2),p)$ for every $p \in \mathcal{P}$, then
  $z_1=z_2$ and $\mathfrak{Y}_{\Sigma,x_1}(p,0)=\mathfrak{Y}_{\Sigma,x_2}(0,p)$ for all 
  $p \in \mathcal{P}$. The latter implies that $x_1=x_2$ by observability of $\Sigma$.
  Conversely, if \eqref{eq:bilin:aux} is observable, but there exists 
  $x_1 \ne x_2$ such that $\mathfrak{Y}_{\Sigma,x_1}(p,0) = \mathfrak{Y}_{\Sigma,x_2}(0,p)$ for all
  $p \in \mathcal{P}$, then 
  $s((x_1,0),p)=s((x_2,0),p)$ for all 
  $p \in \mathcal{P}$. The latter contradicts to observability of \eqref{eq:bilin:aux}.

 We argue that there exists a $t_{\mathrm o} > 0$, $p_{\mathrm o} \in \mathcal{P}$, such that 
 for any state $(x_i,z_i)$ of \eqref{eq:bilin:aux}, 
  $i=1,2$   $s((x_1,z_1),p_o)=s((x_2,z_2),p_o)$ on $[0,t_{\mathrm o}]$ implies 
  $\forall p \in \mathcal{P}: s((x_1,z_1),p) = s((x_2,z_2),p)$, and in the
 CT case $p_o$ is analytic. If such a $p_{\mathrm o}$ exists, then 
 for any two states $x_1,x_2$  of $\Sigma$, 
 $\mathfrak{Y}_{\Sigma,x_1}(0,p_{\mathrm o}) = \mathfrak{Y}_{\Sigma,x_2}(0,p_{\mathrm o})$ on $[0,t_{\mathrm o}]$ implies 
 $\forall p \in \mathcal{P}: \mathfrak{Y}_{\Sigma,x_1}(0,p)=\mathfrak{Y}_{\Sigma,x_2}(p,0)$.
Indeed, $\mathfrak{Y}_{\Sigma,x_1}(p_{\mathrm o},0) = \mathfrak{Y}_{\Sigma,x_2}(p_{\mathrm o},0)$ on $[0,t_{\mathrm o}]$ implies
 $s((x_1,0),p_{\mathrm o})=s((x_2,0),p_{\mathrm o})$ on $[0,t_{\mathrm o}]$ and hence
$s((x_1,0),p)=s((x_2,0),p)$ for all $p \in \mathcal{P}$, and thus
$\mathfrak{Y}_{\Sigma,x_1}(p,0)=\delta s((x_1,0),p)=\delta s((x_2,0),p)=\mathfrak{Y}_{\Sigma,x_2}(p,0)$ 
for all $p \in \mathcal{P}$.

 For the CT, we can take any $t_{\mathrm o} > 0$, and we take 
 $p_o$ to be the universal input 
 described in
 \cite[Theorem 2.11]{SussmannUnivObs}, when applied to \eqref{eq:bilin:aux}.
 Note that here we view \eqref{eq:bilin:aux} as a system whose inputs $p$ take
 values in the set $\mathbb{P}$.
 It is easy to see that $\mathbb{P}$ satisfies the assumptions of 
 \cite[Theorem 2.11]{SussmannUnivObs}.
 Indeed, $\mathbb{P}$ is a convex set, hence
 $\mathbb{P}$ is contained in the closure of its interior (see \cite[Corollary 2.3.9]{WebsterBook}). Moreover,
 by \cite[Theorem 2.3.5]{WebsterBook} the interior of $\mathbb{P}$ is convex and hence it
 is connected.
 For the DT case , existence of $p_o$ follows by applying the proof of
 \cite[Theorem 5.3]{WangIO}
 to \eqref{eq:bilin:aux}. 
 In fact, below we give a simplified proof along the lines of \cite[Theorem 5.3]{WangIO}.
 For every $p \in \mathcal{P}$, $0 < t \in \mathbb{N}$ define 
 $B_{p,t}=\{ (h_1,h_2) \in \mathbb{R}^{\NX+\NY} \times \mathbb{R}^{\NX+\NY} \mid \forall s \in [0,t]: s(h_1,p)(s)=s(h_2,p)(s) \}$.
 Clearly, $B_{p,t}$ is a linear subspace of $\mathbb{R}^{\NX+\NY} \times \mathbb{R}^{\NX+\NY}$.
 Let $p^{*},t^{*}$ be such that $t^{*} > 0$ and $\dim B_{p^*,t^*} \le \dim B_{p,t}$ for any
 $p \in \mathcal{P}$, $t > 0$. Such a $p^{*},t^{*}$ will always exist
 since $\dim B_{p,t} \le 2(\NX+\NY)$ is always finite.
 We claim that $B_{p^{*},t^{*}}=\{(h,h)  \mid \mathbb{R}^{\NX+\NY} \}$.
 hence $p_o=p^{*}$, $t_f=t^{*}$ is the desired input. Assume the contrary, i.e. 
 there exists $(h_1,h_2) \in B_{p^{*},t^{*}}$ such that  $h_1 \ne h_2$.
 Let $(\hat{x}_i,\hat{z}_i)$ be the state of \eqref{eq:bilin:aux} at $t^{*}$, 
 if the initial state is $h_i$ and the input is $p^{*}$ and assume that
 $h_i=(x_i,z_i)$, $i=1,2$.
 Let $A(p^*(s))=\sum_{q \in Q} p^{*}_q(s)A_q$ for any $s \in \mathbb{N}$.
 Then $\hat{x}_i=(A(p^*(t^{*}-1))A(p^*(t^{*}-2))\cdots A(p^*(0))x_i$, $i=1,2$.
 Note that $z_1=s(h_1,p^*)(0)=s(h_2,p^*)(0)=z_2$, and hence
 $h_1 \ne h_2$ implies $x_1 \ne x_2$.
 Since by the assumptions of the theorem, $A(p^*(s))$ is invertible for all $s \in \mathbb{N}$,
 it then follows that $\hat{x}_1 \ne \hat{x}_2$ and hence 
 $(\hat{x}_1,\hat{z}_1) \ne (\hat{x}_2,\hat{z}_2)$.
 From the observability of \eqref{eq:bilin:aux} it then follows that
 there exists $\tau >0$, $\hat{p} \in \mathcal{P}$ such that
 $s((\hat{x}_1,\hat{z}_1),\hat{p})(\tau) \ne s((\hat{x}_2,\hat{z}_2),\hat{p})(\tau))$.
 Hence, for $p \in \mathcal{P}$ defined by 
 \( 
      p(s)=\left\{\begin{array}{rl}
                   p^{*}(s) & \mbox{ if } s \le \tau \\
                   p(s-\tau) & \mbox{ if } s > \tau 
           \end{array}\right., 
 \)   
 $(h_1,h_2) \notin B_{p,t^*+\tau}$. Note that $B_{p,t^*+\tau} \subseteq B_{p^*,t^{*}}$.
 Hence, $\dim B_{p,t^*+\tau} < B_{p^*,t^{*}}$, which is a contradiction.

\end{proof}
\begin{proof}[Proof of Corollary \ref{min:compare:col11}]
 Choose $p_o$ as in Theorem \ref{min:compare:col1}. The statement follows
 from the definition of complete observability for LTV systems.
\end{proof}
\begin{proof}[Proof of Corollary \ref{min:compare:col2}]
 Consider the dual LPV-SSA
 $\Sigma^{T}=(\mathbb{P},\{(A^T_q,C^T_q, B_q^T)\}_{q=0}^{n_p})$.
 If $\Sigma$ is span-reachable from zero, 
then $\rank \mathcal{R}_{\NX-1}=n_x$,
 where $\mathcal{R}_{\NX-1}$ is the ($\NX-1$)-step extended reachability matrix of $\Sigma$ from $0$.
 Let $\mathcal{O}_{\NX-1}$ be the ($\NX-1$)-step extended observability matrix of
 $\Sigma^T$. It is clear that $\mathcal{O}_{\NX-1}=\mathcal{R}_{\NX-1}^T$, and hence
 $\rank \mathcal{O}_{\NX-1}=\NX$ and thus $\Sigma^T$ is observable.
 From Corollary \ref{min:compare:col11} it follows that there exist
 $t_{\mathrm o} > 0$ and $p_{\mathrm o}$ such that the LTV system associated with $\Sigma^T$, $p_{\mathrm o}$ is completely observable on $[0,t_{\mathrm o}]$. 
 This LTV system is given
 by matrices $A(t)=A^T(p_{\mathrm o}(t))$, $B(t)=C^T(p_{\mathrm o}(t))$, $C(t)=B^T(p_{\mathrm o}(t))$
 The dual of this LTV system,
 defined by the matrices $A^T(t)=A(p_{\mathrm o}(t))$, $C^T(t)=B(p_{\mathrm o}(t))$, $B^T(t)=C(p_{\mathrm o}(t))$
 is completely controllable on $[0,t_{\mathrm o}]$.
 But this dual LTV system is exactly the LTV system associated with $\Sigma$,
 $p_{\mathrm o}$. Hence, by choosing $p_{\mathrm e}=p_{\mathrm o}$ and $t_{\mathrm e}=t_{\mathrm o}$ the statement of the corollary holds.
\end{proof}
\begin{proof}[Proof of Theorem \ref{min:compare}]
 We prove the theorem for observability, the 
 statement on span-reachability follows by duality. 

 From Corollary \ref{min:compare:col1} it follows that
 the LTV system obtained from
 $\Sigma$ by setting the scheduling parameter to $p_{\mathrm o}$
 is observable on $[0,t_{\mathrm o}]$
 For CT case, from
 \cite{SilvermanObs} it then follows that
 there exists $k \ge 0$, such that the $k$ step observability
 matrix 
 $\mathrm{O}_k(p_{\mathrm o}(t))$ is such that 
 $\rank{\mathrm{O}_k(p_o(t))} = \NX$ for almost all $t$ on $(0,t_{\mathrm o})$.
 Similarly, for the DT case we get that the $k$-step observability
 matrix $\mathrm{O}_k(p_o)$ has rank $\NX$.
 This then means that if we interpret $\mathrm{O}_k(p)$
 as a matrix whose elements belong to the field
 of meromorphic functions $\mathcal{R}$, then
 $\rank{\mathrm{O}_k(p) }= \NX$, where the rank is now
 interpreted over the field $\mathcal{R}$. Hence,
 from Cayley-Hamilton theorem for matrices over
 $\mathcal{R}$ it follows that the rank of $\mathrm{O}_{\NX-1}(p)$ over
 $\mathcal{R}$ is $\NX$, i.e. $\Sigma$ is structurally observable.

\end{proof}
\begin{proof}[Proof of Theorem \ref{theo:exist}]
 Recall from Corollary \ref{switchediofunction:col1} that the sub-Markov parameters of $\mathfrak{F}$ and $\SWS(\mathfrak{F})$ coincide, i.e. 
 $\theta_{\mathfrak{F}}(s)=\theta_{\SWS(\mathfrak{F})}(s)$, $s \in \Words$.  Moreover, when applied to LSS-SSs, the sub-Markov parameters from Definition \ref{def:grimarkov} coincide with the Markov-parameters of \cite{Pet12,PetCocv11}.

More precisely, the values of $\theta_{\mathfrak{S}(\mathfrak{F})}$ (both in CT and DT) coincide with the Markov-parameters defined in  \cite[Definition 11]{Pet12} of a suitable discrete-time input-output map function  $\hat{\mathfrak{F}}: \mathbb{U}^{\mathbb{N}} \times \mathbb{P}^{\mathbb{N}}_{sw} \rightarrow \mathbb{Y}^{\mathbb{N}}$. In fact, 
$\hat{\mathfrak{F}}$ is defined as $\hat{\mathfrak{F}}(u,p)(t)=\eta^{\mathfrak{F}}_{i_t}(i_0\cdots i_{t-1})+\sum_{j=0}^{t-1} \theta^{\mathfrak{F}}_{i_t,i_{j}}(i_{j+1}\cdots i_{t-1})u(j)$, 
for all $p \in \mathbb{P}^{\mathbb{N}}_{sw}$, $u \in \mathbb{U}^{\mathbb{N}}$, $t \in \mathbb{N}$, 
where $p(k)=e^{i_k}$, $i_k \in \mathbb{I}_0^{n_{\mathrm p}}$, $k=0,\ldots,t$. By Lemma \ref{lem:realiofunction} and \cite[Lemma 1]{Pet12}, an LSS-SS $(\mathbb{P}_{sw}, \{A_i,B_i,C_i,0\}_{i=0}^{\QNUM})$ is a realization of $\SWS(\mathfrak{F})$ (in (CT) or (DT)) if and only if the LSS-SS $(\mathbb{P}_{sw}, \{\hat{A}_i,\hat{B}_i,\hat{C}_i,0\}_{i=0}^{n_{\mathrm p}})$, where $(\hat{A}_0,\hat{B}_0,\hat{C}_0)=(A_0,B_0,C_0)$, $(\hat{A}_i,\hat{B}_i,\hat{C}_i)=(A_i-A_0,B_i-B_0,C_i-C_0)$, $i \in \mathbb{I}_0^{n_{\mathrm p}}$, $i > 0$, is a realization of $\hat{\mathfrak{F}}$.
Notice that $H_{\mathfrak{F}} = H_{\SWS(\mathfrak{F})}$
 and that the former definition of the Hankel-matrix coincides with the one for $\hat{\mathfrak{F}}$  (see \cite[Definition 13]{Pet12}).

 By Theorem ~\ref{th:alpvlsrelation} and Corollary ~\ref{th:alpvlsrelation:col1}, 
 $\mathfrak{F}$ is realizable by an LPV-SSA if and only if 
 $\SWS(\mathfrak{F})$ is realizable by a LSS-SS. 
 From \cite[Theorem 5]{Pet12} it follows that that latter is equivalent to
 $\Rank H_{\mathfrak{F}}=\Rank H_{\SWS(\mathfrak{F})} < +\infty$ 

 Finally, from the proof of Theorem \ref{theo:min} it follows that
 a LPV-SSA $\Sigma$ is a minimal realization of $\mathfrak{F}$ if and only if
 the LSS-SS $\mathfrak{S}(\Sigma)$ is a minimal realization of 
 $\SWS(\mathfrak{F})$. From \cite[Theorem 5]{Pet12} it then follows that $\Rank H_{\SWS(\mathfrak{F})} = \dim \mathfrak{S}(\Sigma)$ and hence $\Rank H_{\mathfrak{F}}=\Rank H_{\SWS(\mathfrak{F})}=\dim \mathfrak{S}(\Sigma)=\dim \Sigma$.
\end{proof}
\begin{proof}[Proof of Theorem \ref{theo:part_real}] 
 From the proof of Theorem \ref{theo:exist} it follows that $H_{\mathfrak{F}}=H_{\SWS(\mathfrak{F})}$ and
 hence $H_{\mathfrak{F}}(n,m)=H_{\SWS(\mathfrak{F})}(n,m)$ for any $n,m \in \mathbb{N}$.
 It is also easy to see that Algorithm \ref{alg0} applied to $H_{\mathfrak{F}}(n,m)=H_{\SWS(\mathfrak{F})}(n,m)$, $m=n+1$,
 coincides with \cite[Algorithm 1]{Pet12} for the Hankel-matrix of $\hat{\mathfrak{F}}$, where $\hat{\mathfrak{F}}$ is the input-output map defined in the proof of Theorem \ref{theo:exist}.
 Hence \cite[Theorem 6]{Pet12} (DT) the following holds. 
 If $\Rank H_{\mathfrak{F}}(n,m)=\Rank H_{\mathfrak{F}}(n,m)=\Rank H_{\mathfrak{F}}(n,m)$ then 
 Algorithm \ref{alg0} returns an LPV-SSS $\Sigma$ and an initial state $x_{\mathrm o}$ such  that  
 $\mathfrak{S}(\Sigma)$ is an $2n+1$ partial realization of $\hat{\mathfrak{F}}$, and hence of $\SWS(\mathcal{F})$ from $x_{\mathrm o}$.
 Since by Corollary \ref{switchediofunction:col1} the sub-Markov parameters of $\mathfrak{F}$ and $\SWS(\mathfrak{F})$ 
 coincide, it then follows that $\Sigma$ is an $2n+1$ partial realization of $\mathfrak{F}$ from $x_{\mathrm o}$.
 If $\Rank H_{\mathfrak{F}}(n,n)=\Rank H_{\mathfrak{F}}$ then 
 $\Rank H_{\mathfrak{F}}(n,n)=\Rank H_{\mathfrak{F}}(n+1,n)=\Rank H_{\mathfrak{F}}(n,n+1)$. In addition, in this case
 $\mathfrak{S}(\Sigma)$ is a minimal realization of $\SWS(\mathfrak{F})$ from $x_{\mathrm o}$.
 From Theorem ~\ref{th:alpvlsrelation} and the proof of
 Theorem \ref{theo:min} it then follows that $\Sigma$ is a minimal realization of
 $\mathfrak{F}$. If $\mathfrak{F}$ has a LPV-SSA realization $\hat{\Sigma}$
 such that $\dim \hat{\Sigma} \le \NX+1$, then
 by Theorem ~\ref{th:alpvlsrelation} $\mathfrak{S}(\hat{\Sigma})$ is a 
 realization of $\SWS(\mathfrak{F})$ and hence by \cite[Theorem 6]{Pet12} 
 $\Rank H_{\SWS(\mathfrak{F})}(n,n)=\Rank H_{\SWS(\mathfrak{F})}$ and hence $\Rank H_{\mathfrak{F}}(n,n)=\Rank H_{\SWS(\mathfrak{F})}(n,n)=\Rank H_{\SWS(\mathfrak{F})}=\Rank H_{\mathfrak{F}}$.
  Finally, that without any conditions, Algorithm \ref{alg0} returns an $n$-moment partial realization follows by adapting the argument of
  \cite[Chapter 10, Proposition 46]{Pet06}. 
\end{proof}

\bibliographystyle{IEEEtran}

\end{document}